\documentclass[a4paper,12pt]{article}

\usepackage{amsthm,amsmath,mathtools}
\usepackage[vmargin=2cm,hmargin=2cm,headheight=14.5pt,top=2cm,headsep=.5cm]{geometry}
\usepackage{hyperref}
\usepackage[utf8]{inputenc}
\usepackage[autobold]{mathfixs}
\usepackage{stackrel}
\usepackage{cases}
\usepackage[charter]{mathdesign}
\DeclareMathSizes{12}{12}{8}{6}

\usepackage{cite}
\usepackage{enumitem}
\usepackage{pgfplots}

\usepackage{tikz-cd}

\newtheoremstyle{ptheorem}{1em}{0em}{\itshape}{}{\bfseries}{.}{.5em}{\thmname{#1}\thmnumber{
		#2}\thmnote{ (\hspace{-1sp}{#3})}}

\theoremstyle{ptheorem}

\newtheorem{thm}{Theorem}[section]
\newtheorem{pro}[thm]{Proposition}
\newtheorem{lem}[thm]{Lemma}
\newtheorem{cor}[thm]{Corollary}

\newtheoremstyle{hdef}{1em}{0em}{}{}{\bfseries}{.}{.5em}{\thmname{#1}\thmnumber{
		#2}\thmnote{ (\hspace{-.01pt}{#3})}}
\theoremstyle{hdef}

\newtheorem{dfn}[thm]{Definition}
\newtheorem{rem}[thm]{Remark}

\newtheorem{exa}[thm]{Example}

\numberwithin{equation}{section}
\numberwithin{figure}{section}

\allowdisplaybreaks

\begin{document}

\title{Stieltjes {analytic} functions and higher order linear differential equations}
\author{Francisco J. Fernández$^{*}$, Ignacio Márquez Albés$^{**\dagger}$ and F. Adri\'an F. Tojo$^{*\dagger}$}
\date{}

\author{Víctor Cora$^{*}$\\
	\normalsize e-mail: victor.cora.calvo@usc.es\\
	F. Javier Fernández$^{*}$\\
	\normalsize e-mail: fjavier.fernandez@usc.es\\
	F. Adri\'an F. Tojo$^{*\dagger}$ \\
	\normalsize e-mail: fernandoadrian.fernandez@usc.es\\ \normalsize \emph{$^{*}$Departamento de Estatística, Análise Matemática e Optimización},\\ \normalsize \emph{Universidade de Santiago de Compostela, 15782, Facultade de Matemáticas, Santiago, Spain.}\\\normalsize
	\\\normalsize \emph{$^{\dagger}$CITMAga, 15782, Santiago de Compostela, Spain}}

\date{}

 \maketitle

\medbreak

\begin{abstract}
		In this work we develop a theory of Stieltjes-analytic functions. We first define the Stieltjes monomials and polynomials and we study them exhaustively. Then, we introduce the $g$-analytic functions locally, as an infinite series of these Stieltjes monomials and we study their properties in depth and how they relate to higher order Stieltjes differentiation. We define the exponential series and prove that it solves the first order linear problem. Finally, we apply the theory to solve higher order linear homogeneous Stieltjes differential equations with constant coefficients.
		\end{abstract}


\medbreak


\section{Introduction}

In recent times, there has been a rise in popularity of the Stieltjes derivative. It turns out that Stieltjes differential equations model more precisely some systems and phenomena than classical ordinary differential equations do, as authors have shown in \cite{LopezPouso2018,Pouso2018,RodrigoTheory}. The reason behind it is that the Stieltjes derivative generalizes times scales \cite{RodrigoTheory}, which is a theory that allows differential equations to have impulsive and stationary behavior. This is particularly interesting when modeling systems that suffer brusque changes or stay latent for some periods of time.

In the literature related to Stieltjes differential equations we can find from foundational works regarding the Stieltjes derivative \cite{PousoRodriguez}, to the study of the first order linear problem along with Picard and Peano type existence results \cite{RodrigoTheory} as well as existence results in another settings \cite{LopezPouso2018,Pouso2018}. In \cite{Onfirst}, for the first time, the authors were able to Stieltjes-differentiate functions at every point of their domain. This allows to successfully define the notion of higher order Stieltjes derivatives and, thus, to consider higher order differential problems \cite{Wronskian,Onfirst}. 

Since we can differentiate functions an arbitrary number of times, we can define the space of $\mathcal{C}^\infty$-differentiable (or $\mathcal{C}^k$-differentiable) functions, as they do in \cite{Onfirst}. A question that comes to mind is: what non trivial examples of $\mathcal{C}^\infty$-differentiable functions can we give? In this work we focus on what it should be a smaller subset, that is, the space of Stieltjes-analytic functions. Note that these functions are not obvious to define. First, we have to clarify what a Stieltjes-monomial is, giving birth to a Stieljtes-polynomial theory, so that then we can define the Stieltjes-analytic functions as an infinite series of these monomials. This is no trivial task, as we have to pay attention to several technical difficulties related to the specific behavior of Stieltjes differentiation.

We then apply this Stieltjes-analytic function theory to find solutions of differential equations, analogously to how it is used in the classical case. We show a way of translating any higher order linear Stieltjes differential equation with constant coefficients to a difference equation on the coefficients of a Stieltjes-analytic function, which always has a solution. We then prove that such Stieltjes-analytic function actually solves the original problem. In particular, we apply this method to the first order linear problem and define the exponential series. We show it is indeed a Stieltjes-analytic function, study its properties, interval of convergence and whether or not it has an analytic continuation to the whole real line. We then compare the exponential series to the exponential function given in \cite{RodrigoTheory,Onfirst}, the solution of the first order linear problem that the literature considers, and show that they are equivalent.

The structure of this work is as follows: In Section~\ref{preliminares} we present some preliminary concepts and definitions needed to prove the statements in the following sections. In Section~\ref{monomios} we introduce the notion of $g$-monomial and $g$-polynomial. We study their properties and relations to monomials associated to other (special) derivators which we will later be able to calculate explicitly. In Section~\ref{analitica} we finally introduce the Stieltjes-analytic functions. We study and compare them with the classical analytic functions. We show some examples of Stieltjes-analytic functions with interesting behaviors as well. Finally, in Section~\ref{exponencial}, we apply the Stieltjes-analytic function theory to differential equations and we show, as an application, a method that solves any higher order linear homogeneous Stieltjes differential equation with constant coefficients as well as some nonhomogeneous cases. We then define the exponential series. We study where it converges, what rules it follows when we change the center point, its relation to some specific exponential series and compare it to the exponential function defined in \cite{RodrigoTheory,Onfirst}.

\section{Preliminaries}
\label{preliminares}
In this Section we provide some basic notions related to the Stieltjes derivative which will be needed later on to introduce the Stieltjes-analytic functions. We recommend the reader to check the cited bibliography, in particular \cite{Onfirst,RodrigoTheory,TesisIgnacio}.
\subsection{Derivators}
\begin{dfn}
	\label{derdef}
	We will call any $g:\mathbb{R}\to \mathbb{R}$ left-continuous non-decreasing function a \emph{derivator}. To fix notation, we will reserve the letter $g$ for derivators. 
\end{dfn}
It is actually very common to find derivators in the literature defined only on an interval ${[a,b]\subset\mathbb R}$. However, we will assume without loss of generality that every derivator is defined on the whole real line. To see this, take a derivator ${\widehat{g}:[a,b]\to \mathbb{R}}$, with $a<b$. We can frame $\widehat{g}$ in Definition~\ref{derdef} by taking
\[  g(x)=\begin{dcases}
	\widehat{g}(b)+x-b,& x> b,\\
	\widehat{g}(x),& x\in[a,b],\\
	\widehat{g}(a)+x-a,&x<a.
\end{dcases}\]  

Given a derivator $g$, we define the pseudometric
\[  d_g(x,y)=\left\lvert g(x)-g(y)\right\rvert \text{, for } x, y \in \mathbb{R}.\]  
\begin{dfn}
	Given any $x\in\mathbb{R}$ and $\varepsilon>0$, we define the \emph{ball of center $x$ and radius $\varepsilon$} as the set
	\[  B_g(x,\varepsilon): =\{ y\in \mathbb{R} \ :\   \ d_g(x,y)<\varepsilon \}. \]  
	Define
	\[  \tau_g:=\{U\subset\mathbb{R} \ :\  \forall x\in U\,\, \exists \varepsilon >0 : B_g(x,\varepsilon)\subset U \}. \]  
\end{dfn}
$\tau_g$ is a topology where the balls are open sets. We denote by $\tau_u$ the usual topology.

From now on, $\mathbb{F}$ will denote either ${\mathbb R}$ or ${\mathbb C}$.
\begin{dfn}
	Given any $X\subset \mathbb{R}$, we say a function $f:X \to \mathbb{F}$ is \emph{$g$-continuous at a point $x\in X$} if 
	\[  \forall \varepsilon >0 \hspace{.15cm}\exists \delta >0 \hspace{.05cm} \text{ such that }\hspace{.05cm} y\in B_g(x,\delta)\cap X \Rightarrow \left\lvert f(x)-f(y)\right\rvert<\varepsilon.\]  
	We say $f$ is \emph{$g$-continuous on $X$} if it is $g$-continuous at every point $x\in X$.
\end{dfn}
\begin{pro}[{\cite[Proposition 3.2]{RodrigoTheory}}]
	\label{continua}
	If $f:X \to \mathbb{F}$ is $g$-continuous on $X$, then
	\begin{enumerate}[noitemsep, itemsep=.1cm]
		\item[~1.] $f$ left-continuous at every point $x\in X$;
		\item[~2.] if $g$ is continuous at $x\in X$, then $f$ is continuous at $x\in X$;
		\item[~3.] if $g$ is constant on some $[\alpha,\beta]$, then $f$ is constant on $[\alpha,\beta]\cap X$.
	\end{enumerate}
\end{pro}
Denote by $\mathcal{C}_g(X,\mathbb{F})$ the set of $g$-continuous maps defined on $X$ that take values on $\mathbb{F}$. We denote by $\mathcal{BC}_g(X,\mathbb{F})$ the subset of bounded $g$-continuous maps. Both these sets are vector spaces.

\begin{pro}
	\label{numerable}
	$\{ B_g(a,r)\ |\ a,r\in\mathbb Q\}$ is a countable basis of $\tau_g$. 
	As a consequence, $\tau_g$ is second countable.
\end{pro}
\begin{proof}
	Fix some $x\in\mathbb{R}$ and $\varepsilon>0$. We will prove that there are $a_x,r_x\in\mathbb Q$ such that $x\in B_g(a_x,r_x)\subset B_g(x,\varepsilon)$. Since $g$ is left-continuous, there exists a sequence $\{a_n\}_{n\in\mathbb N}\subset \mathbb Q$ that converges to $x$ from the left, and is such that
	\[  d_g(a_n,x)=\left\lvert g(a_n)-g(x)\right\rvert\to 0.\]  
	Hence, we can choose some $n\in\mathbb N$ such that for $a_x=a_n$, $d_g(a_x,x)<\frac{\varepsilon}{2}$. Let $r_x\in\mathbb Q$ be such that $d_g(a_x,x)<r_x<\frac{\varepsilon}{2}$. Thus, $x\in B_g(a_x,r_x)$. Now, for all $y\in B_g(a_x,r_x)$, we have
	\[  d_g(x,y)\leq d_g(x,a_x)+d_g(a_x,y)<\varepsilon.\]  
	Then $B_g(a_x,r_x)\subset B_g(x,\varepsilon)$. For any $U\in\tau_g$, we have 
	\[  U=\bigcup_{x\in U} B_g(x,\varepsilon_x),\]  
	where $\varepsilon_x$ is so that $B_g(x,\varepsilon_x)\subset U$. Now, there are $a_x,r_x\in\mathbb Q$ such that $x\in B_g(a_x,r_x)\subset B_g(x,\varepsilon_x)\subset U$ for all $x\in U$. Clearly,
	\[  U=\bigcup_{x\in U} B_g(a_x,r_x).\qedhere\]  
\end{proof}
Denote $g(x^{+})=\lim\limits_{y\to x^{+}}g(y)\in\mathbb{R}$, for any $x\in\mathbb{R}$. Define $\Delta g(x):=g(x^{+})-g(x)$. $\Delta g(x)$ measures the \emph{jump} of $g$ at any given point $x\in\mathbb R$. 

\begin{dfn}
	Given a derivator $g$, define \[  D_g=\{x\in\mathbb{R} \ |\ \Delta g(x)>0 \}\]  
	as the \emph{set of discontinuities} of $g$. Define also
	\[  C_g=\{x\in\mathbb{R} \ |\ g\text{ is constant on } (x-\varepsilon,x+\varepsilon)\text{ for some } \varepsilon>0 \}.\]  
	Note that $C_g$ is open in the usual topology. Hence, we can write
	\[  C_g=\bigcup_{n\in\Lambda}(a_n,b_n),\]  
	where $\Lambda\subset\mathbb N$, $(a_k,b_k)\cap(a_l,b_l)=\emptyset$ if $k\neq l$ and $a_n,b_n\in\mathbb{R}\cup\{+\infty,-\infty\}$ for $n\in\Lambda$. With this notation, we denote $N_g^{-}:=\{a_n\}_{n\in\Lambda}-D_g$, $N_g^{+}:=\{b_n\}_{n\in\Lambda}-D_g$ and $N_g:=N^{-}_g\cup N^{+}_g$. Note that $\{(a_n,b_n)\ |\ n\in\mathbb N\}$ refers to the connected components of $C_g$.
\end{dfn}

\subsection{$g$-derivative}

\begin{dfn}[{\cite[Definition 3.1]{Onfirst}}]
	\label{derivadadef}
	Let $a,b\in\mathbb{R}$ be such that $a<b$, $a\notin N^{-}_g$ and $b\notin D_g\cup N^{+}_g\cup C_g$. We define the \emph{Stieltjes derivative} or $g$-\emph{derivative} of a function $f:[a,b]\to\mathbb{F}$ at a point $x\in[a,b]$ as
	\begin{equation}
		\label{derivada2}
		f'_g(x)= \begin{dcases}
			\lim\limits_{y\to x^+}\frac{f(y)-f(x)}{g(y)-g(x)},& x\in D_g,\\[  1em]
			\lim\limits_{y\to x}\hspace{.175cm}\frac{f(y)-f(x)}{g(y)-g(x)}, & x\notin D_g\cup C_g,\\[  1em]
			\lim\limits_{y\to b_n^+}\frac{f(y)-f(b_n)}{g(y)-g(b_n)}, & x\in (a_n,b_n)\subset C_g,
		\end{dcases}
	\end{equation}
	where $(a_n,b_n)$ is a connected component of $C_g$. Suppose $x$ falls in the second case, it could be that $x\in N_g$, then we have to understand the corresponding limit as follows
\begin{equation*}
		f'_g(x)= \begin{dcases}
			\lim\limits_{y\to x^+}\frac{f(y)-f(x)}{g(y)-g(x)},& x\in N^{+}_g,\\[  1em]
			\lim\limits_{y\to x^-}\frac{f(y)-f(x)}{g(y)-g(x)},& x\notin N^{-}_g. 		
			\end{dcases}
\end{equation*}

	Let $\Omega\subset \mathbb{R}$ be an open { set} of the usual topology that satisfies
	\begin{equation}
		\label{2222221354324}\forall x\in\Omega \text{ such that } x\in(a_n,b_n)\subset C_g, { \text{ we have that }} b_n\in\Omega,\end{equation}
	where $(a_n,b_n)$ is a connected component of $C_g$. Then, for $f:\Omega\to\mathbb{F}$ and $x\in\Omega$, we define the Stieltjes derivative of $f$ at $x$ following~\eqref{derivada2}.
\end{dfn}

Observe that Definition~\ref{derivadadef} assumes functions have domains where the $g$-derivative is well-defined at all points. Note that we can always define the Stieltjes derivative of a function $f$ at a point $x\in \mathbb R$ as long as $f$ is defined in some neighborhood $(x-\varepsilon,x+\varepsilon)$ of $x$ and $x\notin C_g$, just following~\eqref{derivada2}. Thus, we may say that some function $f$ defined on a neighborhood of a point $x\notin C_g$ is $g$-differentiable at $x$ without its domain satisfying the hypothesis of Definition~\ref{derivadadef}. We strongly recommend the reader to explore the details of this definition as presented in \cite[Definition 3.1]{Onfirst}, \cite[Remarks~3.2 and~3.3]{Onfirst} and \cite[Definition 3.7]{Onfirst}.

There are examples of domains where the Stieltjes derivative is just not defined at all points. In fact, $\mathbb R$ can be one of those domains as we will see in Example~\ref{55886}.
\begin{rem}
	It follows from Definition~\ref{derivadadef} that, for $x\in D_g$, $f'_g(x)$ exists if and only if $f(x^+)$ exists and we have that 
	\[  f'_g(x)=\frac{f(x^+)-f(x)}{\Delta g(x)}. \]  
\end{rem}

\begin{pro}
	\label{opder}
	Fix $a,b\in\mathbb{R}$ such that $a<b$, $a\notin N^{-}_g$ and $b\notin D_g\cup N^{+}_g\cup C_g$. Given $x\in [a,b]$, denote

	\[  x^*=\begin{cases}
		x & \text{ if } x\notin C_g,\\
		b_n & \text{ if } x\in (a_n,b_n)\subset C_g.

	\end{cases} \]  
	Then, given $f_1,f_2:[a,b]\to \mathbb{F}$ $g$-differentiable at $x$, we have that
	\begin{enumerate}[noitemsep, itemsep=.1cm]
		\item[~1.] $\lambda_1f_1+\lambda_2f_2$ is $g$-differentiable at $x$ for any $\lambda_1,\lambda_2\in\mathbb{F}$ and
		\[  (\lambda_1f_1+\lambda_2f_2)'_g(x)= \lambda_1(f_1)'_g(x)+\lambda_2(f_2)'_g(x).\]  

		\item[~2.] $f_1f_2$ is $g$-differentiable at $x$ and
		\[  (f_1f_2)'_g(x)= (f_1)'_g(x)f_2(x^*)+f_1(x^*)(f_2)'_g(x)+(f_1)'_g(x)(f_2)'_g(x)\Delta g(x^*).\]  
	\end{enumerate}
	If $f_1$ and $f_2$ are defined on a neighborhood of $x\notin C_g${, then $(1)$ and $(2)$ are satisfied} (with $x^*=x$). 
\end{pro}

Proofs of Proposition~\ref{opder} can be found in \cite[Proposition 3.9]{Onfirst} and \cite[Proposition 3.13]{TesisIgnacio}. The proof is reduced to computing the limit~\eqref{derivada2} in each case. Suppose $f_1$ and $f_2$ are $g$-continuous on $[a,b]$ then, following Proposition~\ref{continua}, we obtain 
\[  f(x)=f(x^*) \text{ for } x\in[a,b],\]  
simplifying point 2 in Proposition~\ref{opder}. Note that $\Delta g$ may not be left-continuous and hence not $g$-continuous --cf. \cite[Proposition 3.1]{Fernandez2022}. 

In the classical case, if a function $f$ is differentiable at a point $x\in\mathbb R$, then it is continuous at that point. As seen in \cite[Remark 3.3]{Onfirst}, in the case of the Stieltjes derivative, the g-differentiability of a function only guarantees the $g$-continuity at points $x\notin N_g \cup C_g \cup D_g$.

\begin{dfn}
	\label{cinfinito}
	Fix $a,b\in\mathbb{R}$ such that $a<b$, $a\notin N^{-}_g$ and $b\notin D_g\cup N^{+}_g\cup C_g$. We say $f:[a,b]\to\mathbb{F}$ belongs to $\mathcal{C}^1_g([a,b],\mathbb{F})$ if the followings conditions are met
	\begin{enumerate}[noitemsep, itemsep=.1cm]
		\item[~1.] $f\in\mathcal{C}_g([a,b],\mathbb{F})$, 
		\item[~2.] $\exists f'_g(x)$ for all $x\in[a,b]$ and $f'_g\in\mathcal{C}_g([a,b],\mathbb{F}).$ 
	\end{enumerate}
	Given $k\in\mathbb N,\,k>1$, we say $f:[a,b]\to\mathbb{F}$ belongs to $\mathcal{C}^k_g([a,b],\mathbb{F})$ if, recursively,
	\begin{enumerate}[noitemsep, itemsep=.1cm]
		\item[~1.] $f\in\mathcal{C}_g([a,b],\mathbb{F})$, 
		\item[~2.] $\exists f'_g(x)$ for all $x\in[a,b]$ and $f'_g\in\mathcal{C}^{k-1}_g([a,b],\mathbb{F}).$ 
	\end{enumerate}
	We define
	\[  \mathcal{C}^{\infty}_g([a,b],\mathbb{F}):=\bigcap_{k\in\mathbb N}\mathcal{C}^k_g([a,b],\mathbb{F}).\]  
	For $k\in\mathbb N\cup\{\infty\}$, we also define 
	\[  \mathcal{BC}^{k}_g([a,b],\mathbb{F}):=\{f\in\mathcal{C}^k_g([a,b],\mathbb{F})\ |\ f^{(n)}_g\in\mathcal{BC}_g([a,b],\mathbb{F}), \forall n=0,\dots,k\}.\]  
	Analogously, we define the same sets for $\Omega\subset\mathbb{F}$ an open {set} on the usual topology satisfying~\eqref{2222221354324}.
\end{dfn}
Note that, thanks to Proposition~\ref{opder}, all of the above are vector spaces.

\subsection{Lebesgue-Stieltjes Integral}

Throughout the paper we will work with Lebesgue-Stieltjes integrals. The usual way of constructing {Lebesgue–Stieltjes measures} through a non-decreasing map is applying Caratheodory's extension theorem. Here we present the theorem directly applied to derivators. For a full statement and proof see \cite[Theorems~1.3.2 to 1.3.6]{Athreya}. For more details in the derivator's case see \cite[Example~1.46]{TesisIgnacio}. We will denote by $\mathcal{B}$ the Borel $\sigma$-algebra relative to $\tau_u$, the usual topology of the real line.
\begin{thm}[Caratheodory's extension theorem]
Let $g:\mathbb{R}\to\mathbb{R}$ be a derivator and \\$\mu_g^*:\mathcal{P}(\mathbb{R})\to [0,+\infty]$ given by
\[  \mu_g^*(A)=\inf\left\{ \sum_{n\in\mathbb N}(g(b_n)-g(a_n))\,:\, A\subset \bigcup_{n\in\mathbb N}[a_n,b_n),\, \{[a_n,b_n)\}_{n\in\mathbb N}\subset\mathcal{C}\right\}\]  
where
\[  \mathcal{C}=\{[a,b)\,:\, a,b\in\mathbb{R},\, a<b \}.\]  
Then $\mu_g^*$ is an outer measure, the set 
\[  \mathcal{M}_g=\{ A\in\mathcal{P}(\mathbb{R})\,\ |\ \, \mu_g^*(E)=\mu_g^*(E\cap A)+\mu_g^*(E-A)),\,\,\forall E\in\mathcal{P}(\mathbb{R}) \}\]  
is a $\sigma$-algebra and the restriction $\mu_g=\mu_g^*|_{\mathcal{M}_g}$ is a measure on $\mathcal{M}_g$. In particular, $\mathcal{C}\subset\mathcal{M}_g$ and $\mathcal{B}\subset\mathcal{M}_g$. If $\mu'$ is a measure on $\mathcal{B}$ such that $\mu'=\mu_g$ on $\mathcal{C}$, then $\mu'=\mu_g$ on $\mathcal{B}$. 
\end{thm}
Considering $g=\operatorname{Id}$, we recover the classical construction of the Lebesgue measure. Furthermore, \[  \mu_g([a,b))=g(b)-g(a),\]  for all $a,b\in\mathbb{R}$ such that $a<b$, for more details see \cite[Example 1.46]{TesisIgnacio}. For any $x\in\mathbb{R}$, we have \[  \mu_g(\{x\})=\Delta g(x).\]  

\begin{dfn}
Let $X\in\mathcal{M}_g$ and consider the measure space $(X,{\mathcal{M}_g}|_{X},\mu_g)$. Given any function $f:X\to\mathbb{F}$ we say it is: 
\begin{enumerate}[noitemsep, itemsep=.1cm]
	\item[~1.] \emph{$g$-measurable}, if $f^{-1}(U)\in\mathcal{M}_g$, for all $U\in \mathcal{B}$.
	\item[~2.] \emph{$g$-integrable} or $f\in\mathcal{L}^1_{\mu_g}(X,\mathbb{F})$, if it is $g$-measurable and
	\[  \int_X \left\lvert f\right\rvert\operatorname{d}\mu_g<\infty.\]  
\end{enumerate}
\end{dfn}

\begin{pro}
Let $X\in\mathcal{B}$. If $f:X\to\mathbb{F}$ is $g$-continuous on $X$, then $f^{-1}$ takes Borel sets onto Borel sets, in particular, $f$ is $g$-measurable.
\end{pro}
\begin{proof}
Let $U\in\tau_u$ be an open subset of $\mathbb{F}$. Since $f$ is $g$-continuous, we have that $f^{-1}(U)$ is an open set of $\tau_g$ intersected with $X$. From Proposition~\ref{numerable} we know that $f^{-1}(U)$ is a countable union of balls intersected with $X$. Balls are intervals and hence Borel sets, so we have that $f^{-1}(U)\in\mathcal{B}$. Since $\mathcal{B}$ is the smallest $\sigma$-algebra that contains $\tau_u$, we have the result.
\end{proof}

Note that if a function $f:X\to\mathbb{F}$ is such that $f^{-1}$ takes Borel sets onto Borel sets then it is $g$-measurable for any given derivator $g$.

\begin{thm}[Fundamental Theorem of Calculus for the Lebesgue–Stieltjes integral]
Let $a,b\in\mathbb{R}$ such that $a<b$ and $w:[a,b]\to\mathbb{F}$. Then the following concepts are equivalent:
\begin{enumerate}[noitemsep, itemsep=.1cm]
\item[~1.] The map $w$ is \emph{$g$-absolutely continuous}, that is, for each $\varepsilon >0$, there exists some $\delta>0$ such that, for any family $\{(a_n,b_n)\}^{m}_{n=1}$ of pairwise disjoint open subintervals of $[a,b]$, \[  \sum_{n=1}^m(g(b_n)-g(a_n))<\delta \Rightarrow \sum_{n=1}^m \left\lvert w(b_n)-w(a_n)\right\rvert<\varepsilon.\]  
We denote this as $w\in\mathcal{AC}_g([a,b],\mathbb{F})$.
\item[~2.] The map $w$ satisfies the following properties:
\begin{enumerate}[noitemsep, itemsep=.1cm]
	\item[~(a)] There exists $w'_g(x)$ for all $x\in[a,b)$, except on a $g$-measurable set of null $\mu_g$-measure.
	\item[~(b)] $w'_g\in\mathcal{L}_g^1([a,b),\mathbb{F})$.
	\item[~(c)] For all $x\in[a,b]$, \[  w(x)-w(a)=\int_{[a,x)}w'_g\operatorname{d}\mu_g.\]  
\end{enumerate}
\end{enumerate}
\end{thm}

A more general result can be seen in \cite[Theorem 2.71]{TesisIgnacio}. This same statement appears on \cite[Theorem 5.1]{RodrigoTheory}. For a proof, see \cite[Theorem 5.4]{PousoRodriguez}.

Using the $g$-integrability of $g$-continuous functions, we obtain the following result.
\begin{pro}
\label{primitiva}
Let $a,b\in\mathbb{R}$ be such that $a<b$ and $f\in\mathcal{BC}_g([a,b),\mathbb{F})$. Then $f$ is $g$-integrable on $[a,b)$ and the map
\[  F: x\in[a,b]\to F(x)=\int_{[a,x)}f\operatorname{d}\mu_g\]  
is $g$-continuous and bounded on $[a,b]$.
\end{pro}

\begin{thm}[{\cite[Theorem 2.4]{PousoRodriguez}}]
\label{derintstieljtes}
Let $a,b\in\mathbb{R}$ be such that $a<b$ and $f\in\mathcal{L}^1_{\mu_g}([a,b),\mathbb{F})$. Consider the function $F:[a,b]\to \mathbb{F}$ given by
\[  F: x\in[a,b]\to F(x)=\int_{[a,x)}f\operatorname{d}\mu_g.\]  
Then, there exists $N\subset[a,b)$ such that $\mu_g(N)=0$ and $F'_g(x)=f(x)$ for all $x\in[a,b)-N$.
\end{thm}

\begin{pro}
\label{primder}
Let $a,b\in\mathbb{R}$ be such that $a<b$ and $f\in\mathcal{BC}_g([a,b],\mathbb{F})$. Then the map
\[  F: x\in[a,b]\to F(x)=\int_{[a,x)}f\operatorname{d}\mu_g\]  
satisfies $F'_g(x)=f(x)$ for all $x\in(a,b)-C_g$. Besides, if $a\notin N^{-}_g$ and $b\notin D_g\cup N^{+}_g\cup C_g$, then $F'_g(x)=f(x)$ for all $x\in[a,b]$ and $F\in\mathcal{BC}^1_g([a,b],\mathbb{F})$.
\end{pro}
Essentially, this result is proven in \cite[Lemma 3.14]{Onfirst}. See how Proposition~\ref{primder} and Theorem~\ref{derintstieljtes} are so closely related, since $\mu_g(C_g)=\mu_g(N_g)=0$.

\begin{pro}[{\cite[Proposition 1.52]{TesisIgnacio}}]
\label{sumamedidas}
Let $g_1,g_2:\mathbb{R}\to\mathbb{R}$ be two derivators. Define ${g:\mathbb{R}\to\mathbb{R}}$ as
\[  g(x)=g_1(x)+g_2(x),\quad x\in\mathbb{R}.\]  
Then $g$ is also a derivator and \[  \mu_g^*(E)=\mu_{g_1}^*(E)+\mu_{g_2}^*(E), \quad E\in\mathcal{P}(\mathbb{R}).\]  In particular, we have that any subset $g_1$ and $g_2$-measurable is $g$-measurable and
\[  \mu_g(E)=\mu_{g_1}(E)+\mu_{g_2}(E), \quad E\in\mathcal{M}_{g_1}\cap\mathcal{M}_{g_2}.\]  
\end{pro}

If we restrict to Borel sets, $\mu_g=\mu_{g_1}+\mu_{g_2}$. This can help us compute integrals over $\mu_g$. Let $X\in\mathcal{B}$ and $f:X\to\mathbb{F}$ a $g_1$ and $g_2$-measurable map, then $f$ is $g$-integrable if and only if is $g_1$ and $g_2$-integrable and
\[  \int_{X}f\operatorname{d}\mu_g=\int_{X}f\operatorname{d}(\mu_{g_1}+\mu_{g_2})=\int_{X}f\operatorname{d}\mu_{g_1}+\int_{X}f\operatorname{d}\mu_{g_2}.\]  

\section{$g$-Monomials}
\label{monomios}
In this section we define the $g$-monomials and $g$-polynomials and present some interesting properties. We also compute the $g$-monomials explicitly in the case where $g$ is either a continuous or a totally discontinuous derivator. Finally, in Theorem~\ref{ggcgb}, we show that any $g$-monomial is in fact a combination of monomials of these two type of derivators, continuous or totally discontinuous. This relationship amongst monomials will have an impact as well on the exponential series, as we show in Section~\ref{exposeries}.

\subsection{Basic notions}

We recall that, in the usual case, a function $f$ is analytic on an open subset $\Omega$, if $\forall \,x_0\in\Omega$, there exists $\varepsilon>0$ and $\{a_n\}_ {n\in\mathbb N}\subset \mathbb{R}$ such that \[  f(x)=\sum_{n\in\mathbb N}a_n(x-x_0)^n\]  for $x\in (x_0-\varepsilon,x_0+\varepsilon)$, where the convergence of the series is absolute and uniform. In this case, $f$ is $\mathcal{C}^{\infty}$-differentiable on $\Omega$ and also \[  a_n=\frac{f^{(n)}(x_0)}{n!}\]  for all $n\in\mathbb N$, see \cite[Chapter 1]{Krantz1992}. In a certain way, an analytic function is just an infinite sum of monomials or a polynomial of infinite degree. Whereas, in the usual case, polynomials represent the regular function prototype, for the case of a given derivator $g$, a polynomial needs not even be $g$-continuous. If we want to define the concept of a Stieltjes-analytic function, we have to look for a series of functions that are as regular as possible and that maintain their properties when we consider an infinite sum of them. In the classical case we have that
\[  
\int_{[x_0,x)}1\operatorname{d}\mu_{\operatorname{Id}}=x-x_0, \,\, 
2\int_{[x_0,x)}s-x_0\operatorname{d}\mu_{\operatorname{Id}}(s)=(x-x_0)^2,\,\,
3\int_{[x_0,x)}(s-x_0)^2\operatorname{d}\mu_{\operatorname{Id}}(s)=(x-x_0)^3,\,\,
\dots,
\]  
which is a very specific instance of a Peano-Beaker series \cite{Baake2011,Rugh1996}. We can replicate this process and define the $g$-monomials as follows. 

\begin{dfn}
Let $g:\mathbb{R}\to\mathbb{R}$ be a derivator and fix some $x_0\in\mathbb{R}$. We define $g_{x_0,0}(x)=1$ for all $x\in\mathbb{R}$. Given any $n\in\mathbb N$, we define $g_{x_0,n}:\mathbb{R}\to\mathbb{R}$ recursively as
\[  
g_{x_0,n}(x)=\begin{dcases}
	n\int_{[x_0,x)}g_{x_0,n-1}\operatorname{d}\mu_g, & x\geq x_0,\\[  1em]
	-n\int_{[x,x_0)}g_{x_0,n-1}\operatorname{d}\mu_g, & x<x_0.
\end{dcases}
\]  

We will call these functions \emph{$g$-monomials centered at $x_0$}, where $x_0$ is called the \emph{center} of $g_{x_0,n}$. We will call the linear combinations of $g$-monomials centered at $x_0$ \emph{$g$-polynomials centered at $x_0$}.

\end{dfn}

{
\begin{rem} Generalizing the concept of monomial by repeatedly integrating the constant function 1 had already been introduced in the context of time scales, see \cite[Section 1.6]{DETS}. However, the framework that we present 
in this work is more general in the sense of given a time scale $\mathbb T$, i.e., a nonempty closed subset of reals, 
we can recover, thanks to \cite[Theorem 3.1]{PousoRodriguez}, the $\Delta$-derivative by use of the Stieltjes differentiation by considering the derivator $g:x\in \mathbb{R} \rightarrow g(x)=\inf \{ s\in \mathbb T , s\geq x\}$. Thus, the monomials in \cite[Section 1.6]{DETS} coincide with the $g$-monomials for this particular choice of $g$.
\end{rem}
}

\begin{rem}
Note that $g_{x_0,n}(x_0)=0$ for all $n\geq1$. Let $a,b \in\mathbb{R}$ be such that $a<x_0<b$. Since $1\in\mathcal{BC}_g([a,b],\mathbb{R})$, Proposition~\ref{primitiva} assures us that $g_{x_0,1}$ is $g$-continuous and bounded on $[a,b]$. In particular, $g_{x_0,1}$ is $g$-integrable and therefore {$g_{x_0,2}$} is well defined. By induction, we have that $g_{x_0,n}$ is well defined, $g$-continuous and bounded on $[a,b]$ for all $n\in\mathbb{N}$. Note that for $x\in[a,x_0]$,
\[  
g_{x_0,n}(x)=n\int_{[a,x)}g_{x_0,n-1}\operatorname{d}\mu_g-n\int_{[a,x_0)}g_{x_0, n-1}\operatorname{d}\mu_g.
\]  
Applying induction again and Proposition~\ref{primder}, it follows that $(g_{x_0,n})'_g(x)=ng_{x_0,n-1}(x)$ for all $x\in(a,b)-C_g$ for $n\in\mathbb N$. Since we have taken arbitrary $a$ and $b$,
\begin{equation}
	\label{monomiosder}
	(g_{x_0,n})'_g(x)=ng_{x_0,n-1}(x),\quad \forall x\in\mathbb{R}-C_g\quad \forall n\in\mathbb N.
\end{equation}
If $a\notin N^{-}_g$ and $b\notin D_g\cup N^{+}_g\cup C_g$, 
\[  (g_{x_0,n})'_g(x)=ng_{x_0,n-1}(x),\quad \forall x\in[a,b],\]  
and $g_{x_0,n}\in\mathcal{C}^{\infty}_g([a,b],\mathbb{R})$, for all $n\in\mathbb N$. In particular, we have that $\Omega=\mathbb{R}$ satisfies condition~\eqref{2222221354324} if and only if $\infty\notin N^{+}_g$ so, in that case,
\[  (g_{x_0,n})'_g(x)=ng_{x_0,n-1}(x),\quad \forall x\in\mathbb{R}\quad \forall n\in\mathbb N.\]  
\end{rem}

The reason why, in general, we cannot assure the $g$-derivative of $g_{x_0,n}$ is $ng_{x_0,n-1}$ for all points of the real line is that there are derivators for which the derivative is simply not well-defined. Consider the following example.

\begin{exa}
\label{55886}
Let $g:\mathbb{R}\to\mathbb{R}$ be defined as
\[  
g(x)= \begin{dcases} 
	-1, & x\leq -1, \\
	0, & x\in(-1,1],\\
	1, & x>1.
\end{dcases}
\]  
%
%
Note that there is no $b\in\mathbb{R}$ such that $b\notin D_g\cup N^{+}_g\cup C_g$, precisely because $D_g\cup C_g=\mathbb{R}$. There are no intervals $[a,b]\subset\mathbb{R}$ like those of Definition~\ref{cinfinito}. For this derivator, $\mathcal{C}^{k}_g$-differentiable functions do not exist, at least in the way we have defined them. Note that, if we follow the idea of Definition~\ref{derivadadef}, we can define the $g$-derivative for points in $(-\infty,1]$. However, we cannot $g$-differentiate at points in $(1,\infty)$. Even so, 
\[  C_g=(-\infty,-1)\cup(-1,1)\cup(1,\infty)\]  
is such that $\mu_g(C_g)=0$ and~\eqref{monomiosder} holds. 
\end{exa}
To shorten the notation, if convenient, we will write
\[  
\int_a^b f \operatorname{d}\mu_g:=\begin{dcases}\int_{[a,b)} f \operatorname{d}\mu_g, & a\leq b,\\
-\int_{[b,a)} f \operatorname{d}\mu_g, & a> b,
\end{dcases}
\]  
for all $a,b\in\mathbb R$.
\subsection{Properties}

We will give a list of properties that will help us to provide some intuition on the $g$-monomials. Note that $g_{x_0,1}(x)=g(x)-g(x_0)$ for all $x\in\mathbb{R}$. To simplify the notation, whenever we do not specify where we center the $g$-monomials, we will assume that we do so at a given point $x_0\in\mathbb{R}$. From now on, $g_{x_0 ,n}\equiv g_n$.

We recover the notion of monomial in the classic sense when $g=\operatorname{Id}$, that is, $\operatorname{Id}_{0,n}(x)=x^n$ for all $x\in\mathbb{R}$ and $n\in\mathbb N$. This can be seen by induction, although we will give a more general proof in Proposition~\ref{gmoncontinua}.

\subsubsection{Some bounds for $g$-monomials}
\begin{lem}
Let $g:\mathbb{R}\to\mathbb{R}$ be a derivator and fix some $x_0\in\mathbb{R}$. We have that:

\begin{enumerate}[noitemsep, itemsep=.1cm]
	\item[~1.] For $x\geq x_0$ and $n\in\mathbb N$, $g_{n}(x)\geq0$.
	\item[~2.] For $x\leq x_0$ and $n\in\mathbb N$, $g_{2n}(x)\geq0$ and $g_{2n-1}(x)\leq0$. 
\end{enumerate}
\end{lem}
\begin{proof}
The result is immediate for $n=1$. Applying induction, suppose that the lemma is true for some $n\in\mathbb N$. Given $x\geq x_0$, by definition,
\[  g_{n+1}(x)=(n+1)\int_{[x_0,x)} g_{n}\operatorname{d}\mu_g. \]  
By induction we know that $g_{n}$ is non-negative at $[x_0,x)$, so $g_{n+1}(x)\geq0$. If $x<x_0$, by definition
\[  g_{n+1}(x)=-(n+1)\int_{[x,x_0)} g_{n}\operatorname{d}\mu_g. \]  
By induction we know that $g_{n}$ has a constant sign on $[x,x_0)$, so $g_{n+1}(x)$ has the opposite sign.
\end{proof}

\begin{lem}[{\cite[Lemma 2.13]{Wronskian}}]
\label{formulapartes}
Given any $w_1,w_2\in\mathcal{AC}_g([a,b],\mathbb{R})$, we have that $w_1w_2\in\mathcal{AC}_g([a,b],\mathbb{R})$ and, for each $t\in[a,b]$,
\begin{align*} 
	 &w_1(t)w_2(t)-w_1(a)w_2(a)\\ = & 
	\displaystyle\int_{[a,t)}(w_1)'_gw_2\operatorname{d}\mu_g+\int_{[a,t)}w_1(w_2)'_g\operatorname{d}\mu_g+\int_{[a,t)}(w_1)'_g(w_2)'_g\Delta g\operatorname{d}\mu_g.
\end{align*}

\end{lem}
 Thanks to equation~\eqref{monomiosder}, it follows that $g_n\in\mathcal{AC}_g([a,b],\mathbb{R})$ for any $a,b\in\mathbb{R}$ such that $a<b$, for all $n\in\mathbb N$. From~\eqref{monomiosder}, we know that $(g_{n})'_g=ng_{n-1}$, except in a $g$-measurable set of null $\mu_g$-measure. We recall that, by definition, $g_n(x_0)=0$ for all $n\geq1$. Then, thanks to Lemma~\ref{formulapartes}, we have that, for all $n\in \mathbb N$ and $x\in\mathbb R$,
\begin{align*}
&\int_{x_0}^x g_{n-k}g_k\operatorname{d}\mu_g =\frac{1}{n-k+1}\int_{x_0}^x(g_{n-k+1})'_gg_k\operatorname{d}\mu_g\\
=&\frac{1}{n-k+1}\left(g_{n-k+1}(x)g_{k}(x)-k\int_{x_0}^x g_{n-k+1}g_{k-1}\operatorname{d}\mu_g-(n-k+1)k\int_{x_0}^x g_{n-k}g_{k-1}\Delta g\operatorname{d}\mu_g\right)
\end{align*}
for all $k\in\{1,\dots,n-1\}$.
\begin{pro}
\label{cotasupder}
Let $g:\mathbb{R}\to \mathbb{R}$ be a derivator and fix some $x_0\in\mathbb{R}$. For $x\geq x_0$ and $n\in\mathbb N$,
\[  0\leq g_{n}(x)\leq g_{n-k}(x)g_k(x)\]  
for all $k\in\{0,\dots,n\}$. In particular, $g_n(x)\leq g_{1}(x)^n$.
\end{pro}
\begin{proof}
Note that for $n=1$ we already have the result. Let us apply induction, suppose the result is true for some $n\in\mathbb N$. Take $k\in\{1,\dots,n-1\}$ (otherwise the statement is trivial). Then, for $x\geq x_0$,
\begin{align*}
	g_{n+1}(x)=&(n+1)\int_{x_0}^xg_{n}\operatorname{d}\mu_g\,\,\leq\,\, (n+1)\int_{x_0}^xg_{n-k}g_{k}\operatorname{d}\mu_g\,\,=\,\,\frac{n+1}{n-k+1}\int_{x_0}^x(g_{n-k+1})'_g g_{k}\operatorname{d}\mu_g\\
	=&\frac{n+1}{n-k+1}\left(g_{n-k+1}(x)g_{k}(x)-k\int_{x_0}^x g_{n-k+1}g_{k-1}\operatorname{d}\mu_g-(n-k+1)k\int_{x_0}^x g_{n-k}g_{k-1}\Delta g\operatorname{d}\mu_g\right)\\
	\leq&\frac{n+1}{n-k+1}\left(g_{n-k+1}(x)g_{k}(x)-k\int_{x_0}^x g_{n}\operatorname{d}\mu_g-(n-k+1)k\int_{x_0}^x g_{n-k}g_{k-1}\Delta g\operatorname{d}\mu_g\right)\\
	=&\frac{n+1}{n-k+1}\left(g_{n-k+1}(x)g_{k}(x)-\frac{k}{n+1} g_{n+1}(x)-(n-k+1)k\int_{x_0}^x g_{n-k}g_{k-1}\Delta g\operatorname{d}\mu_g\right)	
	.
\end{align*}
Thus, we have that
\[  g_{n+1}(x)\leq g_{n-k+1}(x)g_{k}(x)-(n-k+1)k\int_{x_0}^xg_{n-k}g_{k-1}\Delta g\operatorname{d}\mu_g\leq g_{n-k+1}(x)g_{k}(x).\]  
The last inequality follows from the fact that the second addend is negative. \qedhere
\end{proof}

Sadly, the result is not true for $x<x_0$. For non-continuous derivators the $g$-monomials behave much better to the right than to the left. In fact, we will see that we have the reverse inequality to the left.
\begin{pro}
\label{cotainfizq}
Let $g:\mathbb{R}\to \mathbb{R}$ be a derivator and fix some $x_0\in\mathbb{R}$. For $x<x_0$ and $n\in\mathbb N$,
\[  \left\lvert g_{n}(x)\right\rvert\geq \left\lvert g_{n-k}(x)\right\rvert\left\lvert g_{k}(x)\right\rvert\]  
for all $k\in\{0,\dots,n\}$. In particular, $\left\lvert g_n(x)\right\rvert\geq\left\lvert g_{1}(x)\right\rvert^n$.
\end{pro}
\begin{proof}
Again, for $n=1$ the result is immediate. Suppose the result is true for some $n\in\mathbb N$. Take $k\in\{1,\dots,n-1\}$, for $x<x_0$, we have that
\begin{align*}
&	\left\lvert \frac{g_{n+1}(x)}{n+1}\right\rvert=\left\lvert \int_{x_0}^xg_{n}\operatorname{d}\mu_g\right\rvert=\int_{x}^{x_0}\left\lvert g_{n}\right\rvert\operatorname{d}\mu_g\geq\int_{x}^{x_0}\left\lvert g_{n-k}\right\rvert\left\lvert g_{k}\right\rvert\operatorname{d}\mu_g =\left\lvert \int_{x_0}^x g_{n-k}g_{k}\operatorname{d}\mu_g\right\rvert\\
	=&\frac{1}{n-k+1}\left\lvert g_{n-k+1}(x)g_{k}(x)-k\int_{x_0}^x g_{n-k+1}g_{k-1}\operatorname{d}\mu_g-(n-k+1)k\int_{x_0}^x g_{n-k}g_{k-1}\Delta g\operatorname{d}\mu_g\right\rvert\\
	\geq&\frac{1}{n-k+1}\left(\left\lvert g_{n-k+1}(x)g_{k}(x)-(n-k+1)k\int_{x_0}^x g_{n-k}g_{k-1}\Delta g\operatorname{d}\mu_g\right\rvert-k\left\lvert\int_{x_0}^x g_{n-k+1}g_{k-1}\operatorname{d}\mu_g\right\rvert\right)\\
\geq&\frac{1}{n-k+1}\left(\left\lvert g_{n-k+1}(x)g_{k}(x)-(n-k+1)k\int_{x_0}^x g_{n-k}g_{k-1}\Delta g\operatorname{d}\mu_g\right\rvert-\frac{k}{n+1}\left\lvert g_{n+1}(x)\right\rvert\right).
\end{align*}
Then,
\[  \left\lvert g_{n+1}(x)\right\rvert \geq \left\lvert g_{n-k+1}(x)g_{k}(x)-(n-k+1)k\int_{x_0}^x g_{n-k}g_{k-1}\Delta g\operatorname{d}\mu_g\right\rvert\geq \left\lvert g_{n-k}(x)\right\rvert\left\lvert g_{k}(x)\right\rvert. \]  
The last inequality results from the fact that the two addends have the same sign.\qedhere
\end{proof}

Note that $g=\operatorname{Id}$ reaches the bounds of the Propositions~\ref{cotasupder} and~\ref{cotainfizq}. We will now look for a lower bound on the right and an upper bound on the left.
\begin{dfn}
Given $g:\mathbb{R}\to \mathbb{R}$ a derivator. Define $g^{B}:\mathbb{R}\to\mathbb{R}$ as:
\begin{equation*}
	g^B(x)=\begin{dcases}
		\sum_{s\in[0,x)}\Delta g(s), & x>0,\\
		-\sum_{s\in[x,0)}\Delta g(s), & x\leq 0.
	\end{dcases}
\end{equation*}
We have that $g^{B}$ is nondecreasing and left–continuous. We will say that $g^{B}$ is the \emph{discontinuous} or \emph{jump part} of $g$. We say $g$ is \emph{totally discontinuous} when $g=g^B\ne 0$. We define the \emph{continuous part} of $g$ as follows: 
\[  g^{C}(x):=g(x)- g^B(x),\quad \forall x\in\mathbb{R}.\]  
Thus, $g^{C}$ is nondecreasing and continuous (in the usual sense).
\end{dfn}
By definition we have that $g=g^C+g^B$. We can apply then Proposition~\ref{sumamedidas}. In particular, we have, over Borel sets, that $\mu_g=\mu_{g^C}+\mu_{g^B}$ and, therefore, $\mu_{g^C}\leq\mu_g$ and $\mu_{g^B}\leq\mu_g$.
\begin{pro}
Let $g:\mathbb{R}\to \mathbb{R}$ be a derivator and fix some $x_0\in\mathbb{R}$. If $x\geq x_0$, then
\[  0\leq g_n^C(x)\leq g_n(x)\]  
and
\[  0\leq g_n^B(x)\leq g_n(x)\]  
for all $n \in\mathbb N$.
\end{pro}
\begin{proof}
The proof is identical for both derivators, we will only do it for $g^C$. Again, the case $n=1$ is immediate. Suppose the above is true for some $n\in\mathbb N$. Then, for $x\geq x_0$,

\[  \int_{x_0}^xg_n^C\operatorname{d}\mu_{g^C}\leq\int_{x_0}^xg_n^C\operatorname{d}\mu_{g}\leq\int_{x_0}^xg_n\operatorname{d}\mu_{g}. \]  
Thus, $g_{n+1}^C(x)\leq g_{n+1}(x)$.\qedhere

\end{proof}

\begin{pro}
\label{cotasupizq}
Let $g:\mathbb{R}\to \mathbb{R}$ be a derivator and fix some $x_0\in\mathbb{R}$. If $x\leq x_0$, then
\[  \left\lvert g_n(x)\right\rvert\leq n!\left\lvert g_1(x)\right\rvert^n\]  
for all $n \in\mathbb N$.

\end{pro}

\begin{proof}\belowdisplayskip=-12pt
Note that the above is true for $n=1$. Suppose it is also true for some $n\in\mathbb N$. Then, for $x\leq x_0$,
\begin{align*}
	\left\lvert g_{n+1}(x)\right\rvert\leq&(n+1)\int_{[x,x_0)}\left\lvert g_n\right\rvert\operatorname{d}\mu_g\leq(n+1)n!\int_{[x,x_0)}\left\lvert g_1\right\rvert^n\operatorname{d}\mu_g\\
	\leq&(n+1)!\left\lvert g_1(x)\right\rvert^n(g(x_0)-g(x))=(n+1)!\left\lvert g_1(x)\right\rvert^{n+1}.
\end{align*}
\end{proof}
As a summary, gathering the results of the entire section, we have the following corollary.
%
%
\begin{cor}
Let $g:\mathbb{R}\to \mathbb{R}$ be a derivator and fix some $x_0\in\mathbb{R}$. Let any $n\in\mathbb N$:
\begin{enumerate}[noitemsep, itemsep=.1cm]
		\item[~1.] For $x\geq x_0$ and $\star\in\{B,C\}$, \[  0\leq g^\star_{n}(x)\leq g_n(x)\leq g_1(x)^n.\]  
		\item[~2.] For $x< x_0$, \[  \left\lvert g_1(x)\right\rvert^n\leq \left\lvert g_n(x)\right\rvert\leq n!\left\lvert g_1(x)\right\rvert^n.\]  
\end{enumerate}
\end{cor}
The upper bound given in the Proposition~\ref{cotasupizq} is not optimal when $g$ is a continuous derivator as we will see in Proposition~\ref{gmoncontinua}. Nevertheless, as Example~\ref{ejemploeee} shows, there exist derivators that reach the bound.
%
\begin{exa}
\label{ejemploeee}
Let $g:\mathbb{R}\to\mathbb{R}$ be defined as
\[  
g(x)= \begin{dcases} 
	0, & x>-1,\\
	-h, & x\leq-1,
\end{dcases}
\]  
where $h\in\mathbb{R}$ is a positive real number. We have that $g^B=g$ and $g^C=0$. Besides, $D_g=\{-1\}$ and $C_g=\mathbb{R}-\{-1\}$. For $x\in(-1,\infty)$, we have that $g_{0,n}(x)=0$ since $g$ is constantly $0$ at $(-1,\infty)$. For any $n\in\mathbb N$ and $x\leq-1$, we have that
\[  \frac{ g_{0,n+1}(x)}{n+1}=-\int_{[x,0)}g_{0,n}\operatorname{d}\mu=-\int_{\{-1\}}g_{0,n}\operatorname{d}\mu_g=-g_{0,n}(-1)\Delta g(-1).\]  
However, since $\Delta g(-1)=-g(-1)=h$,
\[  g_{0,n+1}(x)=(n+1)g_{0,n}(-1)g(-1)=-(n+1)hg_{0,n}(-1),\]  
so $g_{0,n}(x)=n!(-h)^n$ for all $x\leq-1$. We have $g_{0,n}(x)=n!g_{0,1}(x)^n$ for all $x\leq0$.
\end{exa}

\subsubsection{Center change}

In the classic case, if we want to change the center point of a power series we need to know the relation of the monomials centered at both points (the Binomial theorem). We achieve just that with the following result.
\begin{pro}
\label{relpolprop}
Let $g:\mathbb{R}\to \mathbb{R}$ be a derivator. Fix some $r,s\in\mathbb{R}$ and $n\in\mathbb N$, we have that
\begin{equation}
	\label{relpol}
	g_{r,n}(x)=\sum_{k=0}^n {n\choose k} g_{r,k}(s) g_{s,n-k}(x)
\end{equation}
for any $x\in\mathbb{R}$. 
\end{pro}

\begin{proof}
Take $n=1$ and $x\in\mathbb{R}$, we already have that
\[  g_{r,1}(x)=g(x)-g(r)=g(x)-g(s)+g(s)-g(r)=g_{s,1}(x)+g_{r,1}(s).\]  
We proceed by induction. If the above is true for some $n\in\mathbb N$ then, for all $r,s,x\in\mathbb R$,
\begin{align*}
	\frac{g_{r,n+1}(x)}{n+1}&=\int_{r}^xg_{r,n}\operatorname{d}\mu_g=\int_{r}^sg_{r,n}\operatorname{d}\mu_g+\int_{s}^xg_{r,n}\operatorname{d}\mu_g\\[  .3em]
	&=\frac{g_{r,n+1}(s)}{n+1}+\int_{s}^x\sum_{k=0}^n {n\choose k} g_{r,k}(s) g_{s,n-k}(t)\operatorname{d}\mu_g(t)\\[  .3em]
	&=\frac{g_{r,n+1}(s)}{n+1}+\sum_{k=0}^n {n\choose k} g_{r,k}(s) \int_{s}^xg_{s,n-k}(t)\operatorname{d}\mu_g(t)\\[  .3em]
	&=\frac{g_{r,n+1}(s)}{n+1}+\sum_{k=0}^n {n\choose k} g_{r,k}(s) \frac{g_{s,n-k+1}(x)}{n-k+1}.
\end{align*}
Hence,
\begin{align*}
	g_{r,n+1}(x)&=g_{r,n+1}(s)g_{s,0}(x)+\sum_{k=0}^n \frac{(n+1)n!}{(n-k+1)(n-k)!k!} g_{r,k}(s) g_{s,n-k+1}(x)\\[  .3em]
	&=\sum_{k=0}^{n+1} {n+1\choose k} g_{r,k}(s) g_{s,n+1-k}(x).
\end{align*}
\end{proof}

Proposition~\ref{relpolprop} works as a generalization of the Binomial theorem. If $\operatorname{Id}=g$, expression~\eqref{relpol} tells us that
\[  (x-r)^n=\operatorname{Id}_{r,n}(x)=\sum_{k=0}^n {n\choose k} \operatorname{Id}_{r,k}(s) \operatorname{Id}_{s,n-k}(x)=\sum_{k=0}^n {n\choose k} (s-r)^k (x-s)^{n-k}.\]  

From ~\eqref{relpol} we deduce that any $g$-monomial centered at $r$ can be written as a linear combination of $g$-monomials centered at $s$. In particular, any $g$-polynomial centered at $r$ is a \mbox{$g$-polynomial} centered at $s$. A $g$-polynomial could be a finite linear combination of $g$-monomials centered at different points. In view of Proposition~\ref{relpolprop}, any $g$-polynomial is a linear combination of \mbox{$g$-monomials} centered at a single point. That is, regardless of the $x_0$ chosen, any $g$-polynomial is a finite linear combination of $\{g_{{x_0},n}\}_{n=0}^\infty$.
\subsubsection{$g$-Monomials in the continuous case}
When the derivator is continuous, the monomials have a very reasonable explicit formula.
\begin{pro}
\label{gmoncontinua}
Let $g:\mathbb{R}\to \mathbb{R}$ be a continuous derivator and fix some $x_0\in\mathbb{R}$. Then, for any $n\in\mathbb N$,
\[  g_n(x)=g_{n-1}(x)g_1(x)\]  
for all $x\in\mathbb{R}$. In particular, $g_n(x)=g_1(x)^n$, for 
any $n\in\mathbb N$ and for all $x\in\mathbb{R}$.
%
\end{pro}
\begin{proof}
Suppose the above holds for some $n\in\mathbb N$. Given $x\in\mathbb{R}$, thanks to the Lemma~\ref{formulapartes}, we have that
\begingroup
\allowdisplaybreaks
\begin{align*}
	\frac{g_{n+1}(x)}{n+1}&=\int_{x_0}^xg_n\operatorname{d}\mu_g=\int_{x_0}^xg_{n-1}g_1\operatorname{d}\mu_g=\frac{1}{n}\int_{x_0}^x(g_{n})'_gg_1\operatorname{d}\mu_g\\[  .3em]
	&=\frac{1}{n}\left(g_{n}(x)g_{1}(x)-\int_{x_0}^xg_{n}\operatorname{d}\mu_g-n\int_{x_0}^xg_{n-1}\Delta g\operatorname{d}\mu_g\right)=\frac{1}{n}\left(g_{n}(x)g_{1}(x)-\frac{g_{n+1}(x)}{n+1}\right),
\end{align*}
\endgroup
since $\Delta g=0$. Then
\[  g_{n+1}(x)=g_{n}(x)g_{1}(x). \]  
The last result is trivial for $n=1$. Applying induction we get what we wanted.
\end{proof}
Thus, any $g$-polynomial is just the composition of a classic polynomial with $g$. All $g$-poly\-no\-mi\-als are of the form $p(g(x))$ where $p(x)=\sum_{k=0}^na_kx^k$, $x\in\mathbb{R}$.
%
\subsubsection{$g$-Monomials in the discontinuous case}


We will suppose now that $g$ is a derivator such that $g^C=0$ ($g^B=g$). The latter will allow us to compute the $g$-monomials explicitely. 
\begin{pro}
\label{gmondiscontinua}
Let $g:\mathbb{R}\to\mathbb{R}$ be a derivator such that $g^C=0$. Fix some $x_0\in\mathbb{R}$. Given $n\in\mathbb N$, if $x<x_0$, we denote 

\[  I^n_x=\{\sigma_x: D_g\cap [x,x_0)\to \{0,\dots,n\} \ |\ \sum_{y\in D_g\cap[x,x_0)}\sigma_x(y)=n \}.\]  
Then,
\begin{equation}\label{ssad55332}\frac{g_n(x)}{n!}=(-1)^n\sum_{\sigma_x\in I_x^n} \prod_{y\in D_g\cap[x,x_0)} \Delta g(y)^{\sigma_x(y)}.\end{equation}
If $x>x_0$, define
\[  J_x^n=\{\sigma_x: D_g\cap [x_0,x)\to \{0,1\} \ |\ \sum_{y\in D_g\cap[x_0,x)}\sigma_x(y)=n \}.\]  
Then,
\begin{equation}\label{ssasd55332}\frac{g_n(x)}{n!}=\sum_{\sigma_x\in J_x^n} \prod_{y\in D_g\cap[x_0,x)} \Delta g(y)^{\sigma_x(y)}.\end{equation}
Note that $J_x^n$ represents all possible subsets of $n$ elements in $D_g\cap [x_0,x)$. In particular, we have that, if ${\left\lvert D_g\cap [x_0,x)\right\rvert<n}$, $g_n(x)=0$.
\end{pro}
\begin{proof}
	{Assume $D_g$ is a set of isolated points. The general case is achieved combining Proposition~\ref{glimite} with the isolated points case}. Since $g^C=0$, we have that 
	\[  
	g_1(x)=\begin{dcases}
		\sum_{s\in[x_0,x)}\Delta g(s), & x>x_0,\\
		-\sum_{s\in[x,x_0)}\Delta g(s), & x\leq x_0,
	\end{dcases}
	\]  
	so the formula is true for $n=1$. Suppose that it also works for some arbitrary $n\in\mathbb N$ and proceed by induction. Take $x\in\mathbb{R}$ such that $x<x_0$. Since discontinuity points are a set of isolated points, there are only a finite number of them in $[x,x_0)$. Let $\{x_k\}_{k=1}^{m_x}=D_g\cap[x,x_0)$ ordered from highest to lowest, where $m_x=\left\lvert D_g\cap[x,x_0)\right\rvert\in\mathbb N$. Note that the right hand side of the equality~\eqref{ssad55332} can be written as
	\[  (-1)^n\sum_{j\in I_x^{n*}} \Delta g(x_1)^{j_1}\Delta g(x_2)^{j_2}\cdots\Delta g(x_{m_x})^{j_{m_x}},\]  
	where 
	\[  I_x^{n*}=\{j\in\{0,\dots,n\}^{m_x} \ | \ \sum_{k=1}^{m_x} j_k=n \}.\]  
	Then,
	{
		\begin{align}
	\nonumber		\frac{ g_{n+1}(x)}{(n+1)!}&=-\frac{1}{n!}\int_{[x,x_0)}g_n\operatorname{d}\mu_g=-\frac{1}{n!}\sum_{k=1}^{ m_x} g_n(x_k)\Delta g(x_k)\\
	\label{18974}		&=(-1)^{n+1}\sum_{k=1}^{m_x}\left(\sum_{j\in I_{x_k}^{n*}} \Delta g(x_1)^{j_1}\Delta g(x_2)^{j_2}\cdots\Delta g(x_k)^{j_k}\right)\Delta g(x_k)\\
	\label{3215555}		&=(-1)^{n+1}\sum_{j\in I_x^{n+1*}} \Delta g(x_1)^{j_1}\Delta g(x_2)^{j_2}\cdots\Delta g(x_{m_x})^{j_{m_x}},
		\end{align}

	as there is a one to one correspondence of addends in~\eqref{18974} and addends in~\eqref{3215555}. Note that $m_{x_k}=k$, for $k\in\{1,\dots,m_x\}$.}
%

	Take $x\in\mathbb{R}$ such that $x>x_0$. Let $\{x_k\}_{k=1}^{m_x}=D_g\cap[x_0,x)$ be ordered from lowest to highest, where $m_x=\left\lvert D_g\cap[x_0,x)\right\rvert\in\mathbb N$. We will apply induction for $n\in\mathbb N$. The right side of the equality~\eqref{ssasd55332} can be written as
	\begin{equation*}
		\sum_{j\in J_x^{n*}} \Delta g(x_1)^{j_1}\Delta g(x_2)^{j_2}\cdots\Delta g(x_{m_x})^{j_{m_x}},
\end{equation*} 
	where
	\[  J_x^{n*}=\{j\in\{0,1\}^{m_x} \ | \ \sum_{k=1}^{m_x} j_k=n \}.\]  
	Then,
	{\begin{equation*}
		\begin{aligned}
			\frac{ g_{n+1}(x)}{(n+1)!}&=\frac{1}{n!}\int_{[x_0,x)}g_n\operatorname{d}\mu_g=\frac{1}{n!}\sum_{k=1}^{m_x} g_n(x_k)\Delta g(x_k)\\
			&=\sum_{k=1}^{m_x}\left(\sum_{j\in J_{x_k}^{n*}} \Delta g(x_1)^{j_1}\Delta g(x_2)^{j_2}\cdots\Delta g(x_{k-1})^{j_{k-1}}\right)\Delta g(x_k)\\
			&=\sum_{j\in J_x^{n+1*}} \Delta g(x_1)^{j_1}\Delta g(x_2)^{j_2}\cdots\Delta g(x_{m_x})^{j_{m_x}},
		\end{aligned}
	\end{equation*}
	as the same one to one correspondence holds again.} Here, $m_{x_k}=k-1$, for $k\in\{1,\dots,m_x\}$. Note that, if we assume that there are $m_x$ discontinuities in $[x_0,x)$, for $n\geq m_x+1$, $g_n(y)=0$ for $y\in[x_0,x]$.
\end{proof}

\subsection{Derivator approximation}
\label{1286}

We will prove, in Theorem~\ref{ggcgb}, that, in fact, the $g$-monomials are a combination of products of $g^C$ and $g^B$-monomials. But, for that, it will be necessary for us to be able to approximate any derivator by derivators of which the set of discontinuity points is a set of isolated points. 

Let $g:\mathbb{R}\to\mathbb{R}$ be a derivator. Given $m\in\mathbb N$, we denote $D_g^m=\{x\in\mathbb{R}: \Delta g(x)\geq \frac{1}{m}\}$. Clearly, we have that
\[  D_g=\bigcup_{m\in\mathbb N}D_g^m.\]  
Note that, in fact, $D_g^m$ is a set of isolated points. One way of seeing this is that for all $a,b\in\mathbb{R}$ such that $a<b$, $D_g^m\cap[a,b)$ is a finite set of points. This is easy to prove. Suppose it is not finite, then
\[  \infty\leq\sum_{t\in D_g^m\cap[a,b)}\Delta g(t)\leq \sum_{t\in[a,b)}\Delta g(t).\]  
However, we know that the previous sum is convergent and bounded by $g(b)-g(a)$. Define $g^{B,m}:\mathbb{R}\to\mathbb{R}$ given by
\[  g^{B,m}(x)=\begin{dcases}
	\sum_{t\in D_g^m\cap[0,x)}\Delta g(t), & x>0,\\
	-\sum_{t\in D_g^m\cap[x,0)}\Delta g(t), & x\leq0.
\end{dcases}
\]  
In particular, $g^{B,m}$ is left-continuous, nondecreasing and $g$-continuous. Given $m\in\mathbb N$, we define the $g$-continuous derivator 
\begin{equation}
	\label{gm}
	g^m=g^C+g^{B,m}.
\end{equation}
We have that, for any $x\in\mathbb{R}$,
\[  g(x)-g^m(x)=g^B(x)-g^{B,m}(x)=\begin{dcases}
	\sum_{\substack{t\in [0,x)\\ \Delta g(t)<\frac{1}{m} }}\Delta g(t), & x>0,\\
	-\sum_{\substack{t\in [x,0)\\ \Delta g(t)<\frac{1}{m} }}\Delta g(t), & x\leq0,
\end{dcases}\]  
and, again, $g(x)-g^m(x)$ is a $g$-continuous derivator. Note that, in particular, given $R\in\mathbb{R}$ such that $0<R$, for all $x\in[-R,R]$,
\[  \left\lvert g(x)-g^m(x)\right\rvert\leq\sum_{\substack{t\in [-R,R)\\ \Delta g(t)<\frac{1}{m} }}\Delta g(t).\]  
Proving that the above tends to zero when $m$ tends to infinity, we would obtain that ${g^m\to g}$ uniformly on $[-R,R]$.
\begin{pro}
	Let $g:\mathbb{R}\to\mathbb{R}$ be a derivator. For any $a,b\in\mathbb{R}$ such that $a<b$, we have 
	\[  \lim_{m\to\infty}\sum_{\substack{t\in [a,b)\\ \Delta g(t)<\frac{1}{m} }}\Delta g(t)=0.\]  
\end{pro}
\begin{proof}
	If $D_g\cap[a,b)$ is finite, the above is trivial. Assume $D_g\cap [a,b)$ is infinite, let $\{t_n\}_{n \in\mathbb N}=D_g\cap [a,b)$. Fix some $\varepsilon>0$, since the sum of the jumps is convergent, there exists $n_0\in\mathbb N$ such that \[  \sum_{n=n_0}^\infty \Delta g(t_n)<\varepsilon.\]  Choose $m_0\in\mathbb N$ so that $\frac{1}{m_0}\leq\Delta g(t_n)$ for $n=1,\dots, n_0$. Then, for $m\geq m_0$,
	\[  0\leq\sum_{\substack{t\in [a,b)\\ \Delta g(t)<\frac{1}{m} }}\Delta g(t)\leq\sum_{n=n_0}^\infty \Delta g(t_n)<\varepsilon.\qedhere\]  
\end{proof}

\begin{pro}
	\label{glimite}
	Let $g:\mathbb{R}\to\mathbb{R}$ be a derivator and fix some $x_0\in\mathbb{R}$. Fix $R>0$, for every $n\in\mathbb N$, $g_n^m\to g_n$ 	uniformly on $[x_0-R,x_0+R]$.
\end{pro}
\begin{proof}
	Note that we have just proved the case for $n=1$. Recall that $g-g^m$ is a derivator, so Proposition~\ref{sumamedidas} applies. Suppose the result is true for some $n\in\mathbb N$ and denote $A=[x_0-R,x_0+R]$. For any $x\in A$,
	\begin{equation*}
		\begin{aligned}
			&\frac{\left\lvert g_{n+1}^m(x)-g_{n+1}(x)\right\rvert}{n+1}=\left\lvert \int_{x_0}^xg_n^m\operatorname{d}\mu_{g^m}-\int_{x_0}^xg_n\operatorname{d}\mu_g\right\rvert\\[  .5em]
			=&\left\lvert \int_{x_0}^xg_n^m\operatorname{d}\mu_{g^m}- \int_{x_0}^xg_n\operatorname{d}\mu_{g^m}+ \int_{x_0}^xg_n\operatorname{d}\mu_{g^m}-\int_{x_0}^xg_n\operatorname{d}\mu_g\right\rvert\\[  .5em]
			=&\left\lvert \int_{x_0}^x(g_n^m-g_n)\operatorname{d}\mu_{g^m}- \int_{x_0}^xg_n\operatorname{d}(\mu_g-\mu_{g^m})\right\rvert\leq\int_{x_0}^x\left\lvert g_n^m-g_n\right\rvert\operatorname{d}\mu_{g^m}+ \int_{x_0}^x\left\lvert g_n\right\rvert\operatorname{d}(\mu_g-\mu_{g^m})\\[  .5em]
			\leq&\sup_{A}\left\lvert g_n^m-g_n\right\rvert\left\lvert g^m_1(x)\right\rvert+ \sup_{A}\left\lvert g_n\right\rvert\left\lvert g_1(x)-g^m_1(x)\right\rvert.
		\end{aligned}
	\end{equation*}
	Since $g^m_1$ is bounded in $A$ for every $m\in\mathbb{N}$ and $\{g^m_1\}_{m\in\mathbb N}$ converges uniformly to $g_1$, $\{g^m_1\}_{m\in\mathbb N}$ is uniformly bounded. Choose some $M>\left\lvert g^m_1(x)\right\rvert$ for all $x\in A$ and $m\in\mathbb N$. Then,
	\[  \sup_{A}\left\lvert g_{n+1}^m-g_{n+1}\right\rvert
\leq (n+1)\left(M\sup_{A}\left\lvert g_n^m-g_n\right\rvert+ \sup_{A}\left\lvert g_n\right\rvert\sup_{A}\left\lvert g_1-g_1^m\right\rvert\right).\]  
	Hence, if $C= \sup\limits_{A}\left\lvert g_n\right\rvert$,
	\[  \sup_{A}\left\lvert g_{n+1}^m-g_{n+1}\right\rvert\leq(n+1)\left( M\sup_{A}\left\lvert g_n^m-g_n\right\rvert+ C\sup_{A}\left\lvert g_1-g_1^m\right\rvert\right)\to 0,\]  
	when $m$ tends to infinity.
\end{proof}

We have now a way of approximating the monomials of any derivator from monomials of derivators with a finite number of discontinuities on bounded sets. We will prove that any monomial of any derivator $g$ is in fact a combination of $g^C$ and $g^B$-monomials. For that, a couple of results will be needed.


\begin{pro}
	\label{asd23a15d}
{ Let $g:\mathbb{R}\to\mathbb{R}$ be a derivator, $X\in \mathcal{M}_g$ and $f\in\mathcal{L}^1_{g}(X,\mathbb{R})$. Then,	
}	
	\[  \int_X f\operatorname{d}\mu_{g^B}=\sum_{t\in X\cap D_g}f(t)\Delta g(t).\]  
\end{pro}
{\begin{proof} The proof is a direct consequence of \cite[Lemma 2.3]{Onfirst} since 
$\mathcal{L}^1_{g}(X,\mathbb{R})\subset \mathcal{L}^1_{g^B}(X,\mathbb{R})$.
\end{proof}
}
\begin{lem}
	Let $g:\mathbb{R}\to\mathbb{R}$ be a derivator such that $D_g$ is a set of isolated points and fix some $x_0\in\mathbb{R}$. Then, for all $n\in\mathbb N$, $m\in\mathbb N$ and $a,b\in\mathbb{R}$ such that $a<b$,
	\[  m\int_{[a,b)}g_n^Cg_{m-1}^B\operatorname{d}\mu_{g^B}+n\int_{[a,b)}g_{n-1}^Cg_{m}^B\operatorname{d}\mu_{g^C}=g_m^B(b)g_n^C(b)-g_m^B(a)g_n^C(a).\]  
\end{lem}
\begin{proof}
	Since the set of discontinuity points is a set of isolated points, there must be a finite amount of them in $[a,b)$. Let $\{x_i\}_{i=1}^k=D_g\cap(a,b)$ ordered from lowest to highest and add $x_0=a$ and $x_{k+1}=b$. Thanks to Proposition~\ref{asd23a15d}, it follows that
	\begin{equation}
	\label{templabel2}
	m\int_{[a,b)}g_n^Cg_{m-1}^B\operatorname{d}\mu_{g^B}=m\sum_{i=0}^kg_n^C(x_i)g_{m-1}^B(x_i)\Delta g(x_i).
	\end{equation} 
	Let us calculate the {second integral from the statement of the lemma}. Note that $g^B_m$ is constant on $(x_{i-1},x_i)$ for ${i=1,\dots,k+1}$, because it is $g^B$-continuous. Since $g^C$ is continuous, singletons have null $\mu_{g^C}$-measure, hence,
	\begin{align*}
		n\int_{[a,b)}g_{n-1}^Cg_{m}^B\operatorname{d}\mu_{g^C}&=\sum_{i=1}^{k+1}n\int_{(x_{i-1},x_i)}g_{n-1}^Cg_{m}^B\operatorname{d}\mu_{g^C}=\sum_{i=1}^{k+1}g_{m}^B(x_i)n\int_{(x_{i-1},x_i)}g_{n-1}^C\operatorname{d}\mu_{g^C}\\[  .3em]
		&=\sum_{i=1}^{k+1}g_{m}^B(x_i)(g_n^C(x_i)-g_n^C(x_{i-1})).
	\end{align*}
	We have that , for $i=1,\dots,k+1$,
	\[  g_m^B(x_i)-g_m^B(a)=m\int_{[a,x_i)}g_{m-1}^B\operatorname{d} \mu_{g^B}=m\sum_{j=0}^{i-1}g_{m-1}^B(x_j)\Delta g(x_j).\]  
	We can compute now the second integral,
	\begin{equation}
	\begin{aligned}
	\label{templabel1}
		&n\int_{[a,b)}g_{n-1}^Cg_{m}^B\operatorname{d}\mu_{g^C}=\sum_{i=1}^{k+1}g_{m}^B(x_i)(g_n^C(x_i)-g_n^C(x_{i-1}))\\
		=&\sum_{i=1}^{k+1}\left(m\sum_{j=0}^{i-1}g_{m-1}^B(x_j)\Delta g(x_j)+g_m^B(a)\right)(g_n^C(x_i)-g_n^C(x_{i-1})) \\
		=& m\sum_{j=0}^{k}g_{m-1}^B(x_j)\Delta g(x_j)\sum_{i=j+1}^{k+1}(g_n^C(x_i)-g_n^C(x_{i-1})) + g_m^B(a)(g_n^C(b)-g_n^C(a)) \\
		=&m\sum_{j=0}^{k}g_{m-1}^B(x_j)\Delta g(x_j)(g_n^C(b)-g_n^C(x_{j}))+ g_m^B(a)(g_n^C(b)-g_n^C(a)) .
	\end{aligned}
	\end{equation} 
	Finally, adding~\eqref{templabel2} and~\eqref{templabel1} together, 
	\begin{align*}
		&m\sum_{j=0}^{k}g_{m-1}^B(x_j)\Delta g(x_j)g_n^C(b)+ g_m^B(a)(g_n^C(b)-g_n^C(a))\\
		=&(g_m^B(b)-g_m^B(a))g_n^C(b)+ g_m^B(a)(g_n^C(b)-g_n^C(a))=g_m^B(b)g_n^C(b)-g_m^B(a)g_n^C(a). 
	\end{align*}\end{proof}

We obtain the following straightforward corollary.
\begin{cor}
	\label{pedazocol}
	Let $g:\mathbb{R}\to\mathbb{R}$ be a derivator where the set of discontinuity points is a set of isolated points and fix some $x_0\in\mathbb{R}$. For all $n\in\mathbb N$, $m\in\mathbb N$ and $x\in\mathbb R$,
	\[  m\int_{x_0}^xg_n^Cg_{m-1}^B\operatorname{d}\mu_{g^B}+n\int_{x_0}^xg_{n-1}^Cg_{m}^B\operatorname{d}\mu_{g^C}=g_m^B(x)g_n^C(x).\]  
\end{cor}

Finally, we have what is necessary to prove Theorem~\ref{ggcgb}.
\begin{thm}
	\label{ggcgb}
	Let $g:\mathbb{R}\to\mathbb{R}$ be a derivator and fix some $x_0\in\mathbb{R}$. For all $n\in\mathbb N$ and $x\in\mathbb{R}$,
	\[  g_n(x)=\sum_{k=0}^n {n\choose k} g_{k}^C(x) g_{n-k}^B(x).\]  
\end{thm}
\begin{proof}
	We will prove it first for derivators such that the set of discontinuity points is a set of isolated points. For $n=1$, it is obvious that
	\[  g_1(x)=g_1^C(x)+g_1^B(x),\]  
	for all $x\in\mathbb{R}$. Let us apply induction, suppose the result is true for some $n\in\mathbb N$. Thanks to Proposition~\ref{sumamedidas}, for any $x\in \mathbb R$, 
	\begin{equation}\label{6666666666}\begin{aligned}		
		\frac{g_{n+1}(x)}{n+1}=&\int_{x_0}^xg_n\operatorname{d}\mu_g=\int_{x_0}^x\sum_{k=0}^n {n\choose k} g_{k}^C g_{n-k}^B\operatorname{d}\mu_g=\sum_{k=0}^n {n\choose k}\int_{x_0}^x g_{k}^C g_{n-k}^B\operatorname{d}(\mu_{g^B}+\mu_{g^C})\\
=&\sum_{k=0}^n {n\choose k}\left(\int_{x_0}^x g_{k}^C g_{n-k}^B\operatorname{d}\mu_{g^B}+\int_{x_0}^x g_{k}^C g_{n-k}^B\operatorname{d}\mu_{g^C}\right)\\
=&
\sum_{k=1}^n {n\choose k}
\int_{x_0}^x g_{k}^C g_{n-k}^B\operatorname{d}\mu_{g^B} +
\sum_{k=0}^{n-1} {n\choose k} \int_{x_0}^x g_{k}^C g_{n-k}^B\operatorname{d}\mu_{g^C} \\&
+\int_{x_0}^x g_{n}^B\operatorname{d}\mu_{g^B}
+\int_{x_0}^x g_{n}^C \operatorname{d}\mu_{g^C}\\
=& 
\sum_{k=1}^n \left({n\choose k}
\int_{x_0}^x g_{k}^C g_{n-k}^B\operatorname{d}\mu_{g^B} +
{n\choose k-1} \int_{x_0}^x g_{k-1}^C g_{n-k+1}^B\operatorname{d}\mu_{g^C} 
\right) \\&
+\int_{x_0}^x g_{n}^B\operatorname{d}\mu_{g^B}
+\int_{x_0}^x g_{n}^C \operatorname{d}\mu_{g^C}
.
\end{aligned}
	\end{equation}
	Fix $k\in\{1,\dots,n\}$. Then, 
	\begin{equation}
		\label{77777777}
		{n\choose k}\int_{x_0}^x g_{k}^C g_{n-k}^B\operatorname{d}\mu_{g^B}+{n\choose k-1}\int_{x_0}^x g_{k-1}^C g_{n-k+1}^B\operatorname{d}\mu_{g^C}
	\end{equation}
	equals
	\[  
	\frac{1}{n+1}{n+1\choose k}\left( (n-k+1) \int_{x_0}^x g_{k}^C g_{n-k}^B\operatorname{d}\mu_{g^B}+k\int_{x_0}^x g_{k-1}^C g_{n-k+1}^B\operatorname{d}\mu_{g^C} \right).
	\]  
	Applying Corollary~\ref{pedazocol},~\eqref{77777777} equals
	\begin{equation}
		\label{777777777}
		\frac{1}{n+1}{n+1\choose k}g_k^C(x)g_{n+1-k}^B(x). 
	\end{equation}
	Splitting the sum~\eqref{6666666666} in terms like~\eqref{77777777} and thanks to~\eqref{777777777} we have that
	\begingroup
	\allowdisplaybreaks
	\begin{align*}
		\frac{g_{n+1}(x)}{n+1}&=\sum_{k=0}^n {n\choose k}\left(\int_{x_0}^x g_{k}^C g_{n-k}^B\operatorname{d}\mu_{g^B}+\int_{x_0}^x g_{k}^C g_{n-k}^B\operatorname{d}\mu_{g^C}\right)\\[  .3em]
		&=\frac{1}{n+1}\sum_{k=1}^n{n+1\choose k}g_k^C(x)g_{n+1-k}^B(x)+\int_{x_0}^xg_{n}^B\operatorname{d}\mu_{g^B}+\int_{x_0}^x g_{n}^C\operatorname{d}\mu_{g^C}\\[  .3em]
		&=\frac{1}{n+1}\sum_{k=1}^n{n+1\choose k}g_k^C(x)g_{n+1-k}^B(x)+\frac{g_{n+1}^B(x)}{n+1}+\frac{g_{n+1}^C(x)}{n+1}\\[  .3em]
		&=\frac{1}{n+1}\sum_{k=0}^{n+1}{n+1\choose k}g_k^C(x)g_{n+1-k}^B(x).
	\end{align*}
	\endgroup

	Hence, the result is true for derivators where the set of discontinuities is a set of isolated points.

	Now, let $g:\mathbb{R}\to\mathbb{R}$ be an arbitrary derivator. Take $g^m$ for $m\in\mathbb N$, like in~\eqref{gm}. We have that $D_{g^m}=D^m_g$, $(g^m)^B=g^{B,m}$ and then $g^m$ is derivator such that $D_{g^m}$ is a set of isolated points. Besides, thanks to Proposition~\ref{glimite}, for all $n\in\mathbb N$ and $x\in\mathbb{R}$,
	\[  g^m_n(x)\to g_n(x)\]  
	when $m$ tends to infinity. Fix $n\in\mathbb N$,
	\[  g_n^m(x)=\sum_{k=0}^n {n\choose k} g_{k}^C(x) g_{n-k}^{B,m}(x).\]  
	Note that $(g^B)^m=g^{B,m}\to g^B$, we can apply Proposition~\ref{glimite} so,
	\[  g_{k}^{B,m}(x)\to g_k^B(x)\]  
	for all $k\in\mathbb N$ and $x\in\mathbb{R}$. Then, when $m$ tends to infinity,
	\[  g_n^m(x)=\sum_{k=0}^n {n\choose k} g_{k}^C(x) g_{n-k}^{B,m}(x)\to\sum_{k=0}^n {n\choose k} g_{k}^C(x) g_{n-k}^{B}(x)= g_n(x). \qedhere\]  
\end{proof}

\subsection{More calculus on $g$-monomials}
 We will calculate now a new expression for the $g$-monomials using the integration by parts formula (Lemma~\ref{formulapartes}) and iterative integrals. The purpose of this is to see how the $g$-monomials differ from being a power of $g_1$. Define
\begin{equation}\label{hfunction}
		\begin{aligned}
			\frac{h_{1,k}(x)}{k+1}=&\int_{x_0}^x g_k\Delta g \operatorname{d}\mu_g, \; k\geq 0, \\
			\frac{h_{j+1,k}(x)}{k+j+1}=&\int_{x_0}^x h_{j,k} \operatorname{d}\mu_g,\; k\geq 0,\; j\geq 1.
		\end{aligned}
\end{equation}
We have the following result.

\begin{pro}\label{gexpr12222} Let $g:\mathbb R\to\mathbb R$ be a derivator and $n\in\mathbb N$. For all $x\in\mathbb R$,
	\begin{equation*}
		g_{n}(x)= g_{n-1}(x)g_1(x)-\sum_{j=1}^{n-1} h_{j,n-1-j}(x).
	\end{equation*}
\end{pro}

\begin{proof} The result is trivial for $n=1$. We proceed by induction. Suppose the above is true for some $n\in\mathbb N$. For all $x\in\mathbb R$,
		\[  
			\frac{g_{n+1}(x)}{n+1}=\int_{x_0}^x g_{n}\operatorname{d}\mu_g =\int_{x_0}^x g_{n-1} g_1\operatorname{d}\mu_g -\sum_{j=1}^{n-1} \int_{x_0}^xh_{j,n-1-j} \operatorname{d}\mu_g.
\]  
	Applying Lemma~\ref{formulapartes},
		\[  \begin{aligned}
			\int_{x_0}^x g_{n-1} g_1\operatorname{d}\mu_g=&\frac{1}{n}\left(g_{n}(x)g_{1}(x)-\int_{x_0}^x g_{n}\operatorname{d}\mu_g-n\int_{x_0}^x g_{n-1}\Delta g\operatorname{d}\mu_g\right)\\
			=&\frac{1}{n}\left(g_{n}(x)g_{1}(x)-\int_{x_0}^x g_{n}\operatorname{d}\mu_g- h_{1,n-1}\right).
		\end{aligned}\]  
By expression~\eqref{hfunction},
		\[  \begin{aligned}
			\frac{g_{n+1}(x)}{n+1}=&\frac{1}{n}\left(g_{n}(x)g_{1}(x)-\int_{x_0}^x g_{n}\operatorname{d}\mu_g- h_{1,n-1}\right)-\sum_{j=1}^{n-1} \int_{x_0}^xh_{j,n-1-j}(s) \operatorname{d}\mu_g\\
							=&\frac{1}{n}\left(g_{n}(x)g_{1}(x)-\int_{x_0}^x g_{n}\operatorname{d}\mu_g- h_{1,n-1}\right)-\sum_{j=1}^{n-1} \frac {h_{j+1,n-1-j}(x)}{n}\\
							=&\frac{1}{n}\left(g_{n}(x)g_{1}(x)-\int_{x_0}^x g_{n}\operatorname{d}\mu_g- h_{1,n-1}-\sum_{j=1}^{n-1} h_{j+1,n-1-j}(x)\right)\\
							=&\frac{1}{n}\left(g_{n}(x)g_{1}(x)-\int_{x_0}^x g_{n}\operatorname{d}\mu_g-\sum_{j=1}^{n} h_{j,n-j}(x)\right).
		\end{aligned}\]  
Hence,
\[  g_{n+1}(x)= g_{n}(x)g_1(x)-\sum_{j=1}^{n} h_{j,n-j}(x).\qedhere\]  
\end{proof}

We now present a general formula obtained by applying the recursive expression that we have just computed.

\begin{pro}Let $g:\mathbb R\to\mathbb R$ be a derivator and $n\in\mathbb N$. For all $x\in\mathbb R$,
	\begin{equation*}
		g_{n}(x)=g_1(x)^{n}-\sum_{k=1}^{n-1} g_1(x)^{n-1-k}\sum_{j=1}^k h_{j,k-j}(x).
	\end{equation*}
\end{pro}

\begin{proof} The result is trivial for $n=1$. We proceed by induction again. Suppose the above is true for some $n\in\mathbb N$. Applying Proposition~\ref{gexpr12222}, for all $x\in\mathbb R$,
		\[  \begin{aligned}
			 g_{n+1}(x)=& g_{n}(x)g_1(x)-\sum_{j=1}^{n} h_{j,n-j}(x)\\
					=&\left(g_1(x)^{n}-\sum_{k=1}^{n-1}g_1(x)^{n-1-k}\sum_{j=1}^k h_{j,k-j}(x)\right)g_1(x)-\sum_{j=1}^{n} h_{j,n-j}(x)\\
					=&g_1(x)^{n+1}-\sum_{k=1}^{n-1}g_1(x)^{n-k}\sum_{j=1}^k h_{j,k-j}(x)-\sum_{j=1}^{n} h_{j,n-j}(x)\\
					=&g_1(x)^{n+1}-\sum_{k=1}^{n}g_1(x)^{n-k}\sum_{j=1}^k h_{j,k-j}(x).
		\end{aligned}\]  
\end{proof}

\section{Stieltjes-{analytic} functions}
\label{analitica}

In this section we introduce the Stieltjes-{analytic} functions. We study first some of the properties of series of $g$-monomials. We will try to replicate the classical analytic theory as far as possible. We will see several examples that will limit how far we can go. 

Some studies have tried to develop a analytic function theory in the frame of time scales \cite{timescales}. A comparison between this and our theory is particularly interesting since the Stieltjes derivative generalizes time scales \cite{RodrigoTheory}.
\subsection{$g$-Monomial series}

We can start talking about Stieltjes-{analytic} functions now that we have the $g$-monomials already defined. Following the classical case, we want our functions to be an infinite sum of $g$-monomials, that is,
\[  f(x)=\sum_{n=0}^\infty a_n g_n(x),\]  
for some $\{a_n\}_{n=0}^\infty\subset \mathbb{F}$ and $x\in\mathbb{R}$. As we will see later, finding the set of convergence of a series of $g$-monomials can be more challenging than it seems. It is also interesting to ask whether there is any relationship between an analytic function in the usual sense defined by some coefficients at a fixed point, and the, a priori, Stieltjes-analytic function defined by the same coefficients using the $g$-monomials. We will prove, thanks to Theorem~\ref{ggcgb}, that, under certain hypotheses, the map

\[  
\begin{tikzcd}[row sep=3ex,column sep = 6ex]
	\left\{ \begin{array}{cc}
		\text{Series of } g\text{-monomials} \\
		\text{ Stieltjes-analytic functions}
	\end{array} \right\} \arrow{r} & \left\{ \begin{array}{cc}
		\text{Power series} \\
		\text{ Analytic functions}
	\end{array} \right\} \\ 
	f(x)= \sum\limits_{n=0}^\infty a_n g_n(x) \arrow[mapsto]{r} & \widetilde f(x)=\sum\limits_{n=0}^\infty a_n x^n,		
\end{tikzcd}
\]  
satisfies that
\begin{equation}
	\label{formulaseries}
	f(x)= \sum_{k=0}^\infty\frac{ \widetilde f^{(k)}(g_1^C(x))}{k!}g_{k}^B(x).
\end{equation}
A result that we will often use is the following.

\begin{thm}[{\cite[Theorem 8.3]{Rudin}}]
	\label{sumasij}
	Given any sequence $\{a_{ij}\}_{i,j=0}^\infty\subset\mathbb{F}$. If
	\[  {\sum_{i=0}^\infty \sum_{j=0}^\infty\left\lvert a_{ij}\right\rvert<\infty},\]  
	then
	\[  \sum_{i=0}^\infty\sum_{j=0}^\infty a_{ij}=\sum_{j=0}^\infty\sum_{i=0}^\infty a_{ij}.\]  
\end{thm}

For continuous derivators, since $g_n=g_1^n$, we can do the same analytic function theory as in the classical case. In this case, the convergence of a power series centered at $x_0$ either occurs only on $x_0$ or on a ball of positive radius. We can replicate the same result for continuous derivators, which are going to behave much better than those that have discontinuities. Recall that we continue with the notation $g_{x_0,n}\equiv g_n$.
\subsubsection{Convergence}

In the general case, the convergence on the right hand side or the left hand side of $x_0$ does not imply convergence on the other. In fact, it will be a lot easier for series of $g$-monomials to converge at points at the right side than at the left side, see Proposition~\ref{sucesionizq}. This will force us to work with both sides separately. The next Proposition is similar to \cite[Proposition 1.1.1]{Krantz1992}.
\begin{pro}
	\label{gcontinuoconv}
	Let $g:\mathbb{R}\to\mathbb{R}$ be a continuous derivator and fix some $x_0\in\mathbb{R}$. If
	\[  \sum_{n=0}^\infty a_n g_1(x)^n\]  
	converges for some $x=c\in\mathbb{R}$, then the series converges absolutely on $B_g(x_0,\left\lvert g_1(c)\right\rvert)$. 
\end{pro}
\begin{proof}
	If $g_1(c)=0$, the ball is empty and we can only guarantee convergence on the set $g_1^{-1}(\{0\})$. Suppose $\left\lvert g_1(c)\right\rvert>0$. For $y\in B_g(x_0,\left\lvert g_1(c)\right\rvert)$, 
	\[  \sum_{n=0}^\infty \left\lvert a_n\right\rvert\left\lvert g_1(y)^n\right\rvert=\sum_{n=0}^\infty \left\lvert a_n g_1(c)^n\right\rvert\left\lvert \frac{g_1(y)}{g_1(c)}\right\rvert^n.\]  
	Since the series converges at $c$, the term $\left\lvert a_n g_1(c)^n\right\rvert$ is bounded by some constant $C>0$. {Since $y\in B_g(x_0,\left\lvert g_1(c)\right\rvert)$ we have that $\left\lvert g_1(y)\right\rvert < \left\lvert g_1(c)\right\rvert$, so}
	\[  \sum_{n=0}^\infty \left\lvert a_n\right\rvert\left\lvert g_1(y)^n\right\rvert\leq\sum_{n=0}^\infty C\left\lvert \frac{g_1(y)}{g_1(c)}\right\rvert^n<\infty.\qedhere\]  
\end{proof}
\begin{pro}
	\label{convergenciags}
	Let $g:\mathbb{R}\to\mathbb{R}$ be a derivator and fix some $x_0\in\mathbb{R}$. Take some $c_1,c_2\in\mathbb{R}$ such that $x_0\in(c_1,c_2)$. If
	\[  \sum_{n=0}^\infty a_n g_n(x)\]  
	converges absolutely for $x=c_i$, with $i=1,2$. Then the series converges absolutely and uniformly on $[c_1,c_2]$. Besides, if $M={ \max\limits_{i =1,2}\left\lvert g_1^C(c_i)\right\rvert>0}$, then formula~\eqref{formulaseries} holds for all $x\in[c_1,c_2]\cap B_{g^C}(x_0,M)$.
\end{pro}
\begin{proof}
	Note that for all $n\in\mathbb N$, $\left\lvert g_n\right\rvert$ increases as we move away from $x_0$. Then,
	\[  \sum_{n=0}^\infty \left\lvert a_n\right\rvert \left\lvert g_n(x)\right\rvert\leq \max_{i=1,2}\sum_{n=0}^\infty \left\lvert a_n\right\rvert \left\lvert g_n(c_i)\right\rvert,\]  
	for all $x\in[c_1,c_2]$. Thanks to Weierstrass M-test, the series converges uniformly on $[c_1,c_2]$. Note that, for all $x\in\mathbb{R}$, thanks to Theorem~\ref{ggcgb},
	\[  \left\lvert g_n(x)\right\rvert=\sum_{k=0}^n {n\choose k} \left\lvert g_{k}^B(x)\right\rvert\left\lvert g_{n-k}^C(x)\right\rvert,\]  
	since every addend has the same sign. In particular, we have that
	\[  \left\lvert g^\star_n(x)\right\rvert\leq \left\lvert g_n(x)\right\rvert,\]  
	with $\star\in\{B,C\}$, for all $n\in\mathbb N$ and $x\in\mathbb{R}$. Suppose now that $\left\lvert g_1^C(c_i)\right\rvert>0$ for some $i=1,2$ and denote $M={ \max\limits_{i =1,2}\left\lvert g_1^C(c_i)\right\rvert}$. Applying Proposition~\ref{gcontinuoconv} to the identity { function}, $\{a_n\}_{n\in\mathbb N}$ defines a power series that converges for $\left\lvert x\right\rvert<M$. Define then
	\[  \widetilde f(x)=\sum_{n=0}^\infty a_n x^n.\]  
	In particular, $f$ is a usual analytic function and we can express its derivatives as a power series at $\left\lvert x\right\rvert<M$. For any $x\in[c_1,c_2]$, we have
	\[  \sum_{n=0}^\infty \left\lvert a_n\right\rvert\left\lvert g_n(x)\right\rvert=\sum_{n=0}^\infty \left\lvert a_n\right\rvert\sum_{k=0}^n {n\choose k} \left\lvert g_{k}^B(x)\right\rvert\left\lvert g_{n-k}^C(x)\right\rvert<\infty.\]  
	Thanks to Theorem~\ref{sumasij},
	\begin{equation}
		\label{formulaseries1}
		\begin{aligned}
			\sum_{n=0}^\infty a_n g_n(x)&=\sum_{n=0}^\infty a_n\sum_{k=0}^n g_{k}^B(x) g_{n-k}^C(x)=\sum_{k=0}^\infty\sum_{n=k}^\infty a_n {n\choose k} g_{n-k}^C(x)g_{k}^B(x) \\
			&=\sum_{k=0}^\infty \frac{g_{k}^B(x)}{k!}\sum_{n=k}^\infty a_n \frac{n!}{(n-k)!}g_1^C(x)^{n-k}.
		\end{aligned}
	\end{equation}
	Recall that
	\[  \widetilde f^{(k)}(x)=\sum_{n=k}^\infty a_n \frac{n!}{(n-k)!}x^{n-k},\]  
	for all $\left\lvert x\right\rvert<M$. We obtain, for all $x\in[c_1,c_2]$ such that $\left\lvert g_1^C(x)\right\rvert<M$,
	\[  \sum_{n=0}^\infty a_n g_n(x)= \sum_{k=0}^\infty\frac{\widetilde f^{(k)}(g_1^C(x))}{k!}g_{k}^B(x).\qedhere\]  
\end{proof}


Note that we only needed the absolute convergence of the series of $g$-monomials for us to get to~\eqref{formulaseries1}. Applying both Proposition~\ref{gcontinuoconv} and Proposition~\ref{convergenciags} we conclude the following. 
\begin{cor}
If $g$ is a continuous derivator and
\[  f(x)=\sum_{n=0}^\infty a_n g_1(x)^n\]  
converges at some $c\in\mathbb{R}$ such that $g_1(c)\neq0$, then formula~\eqref{formulaseries} satisfies and
\[  f(x)=\sum_{n=0}^\infty a_ng_1(x)^n=\widetilde f(g_1(x)),\]  
for all $x\in B_g(x_0,\left\lvert g_1(c)\right\rvert)$.
\end{cor}

\begin{rem}
	\label{contydescguay}
	In general, for continuous derivators, the theory behaves like in the classical case. In a certain way, the convergence will also behave well for derivators of which the discontinuity points are isolated points. Suppose we have such a derivator $g$. Fix some $x_0\in\mathbb{R}$. There exists $\delta>0$ such that $g$ contains {at most one} discontinuity in the interval $(x_0-\delta,x_0+\delta)${, namely $x_0$ itself}. In particular, we have
%
%
	\[  g_1^B(x)=\begin{dcases}
		\Delta g(x_0), & x\in(x_0,x_0+\delta),\\
		0, & x\in(x_0-\delta,x_0].
	\end{dcases}
	\]  
	Then, $g^{B}_n{(x)}=0$ for all $n\geq2$ and $x\in(x_0-\delta,x_0+\delta)$, {recall Proposition~\ref{gmondiscontinua}}. Thanks to Theorem~\ref{ggcgb}:
	\[  g_n(x)=\begin{dcases}
		g_1^C(x)^n+n\Delta g(x_0)g_1^C(x)^{n-1}, & x\in(x_0,x_0+\delta),\\
		g_1^C(x)^n, & x\in(x_0-\delta,x_0].
	\end{dcases}
	\]  
	Any $g$-monomial series will converge in a usual neighborhood of $x_0$ if and only if the series of $g^C$-monomials converges. For $x\in(x_0-\delta,x_0+\delta)$, the formula~\eqref{formulaseries} tells us that
	\[  \sum_{n=0}^\infty a_n g_n(x)=f(g_1^C(x))+g_1^B(x)f'(g_1^C(x)).\]  
\end{rem}

\subsubsection{Change of center}

Let us see how the change of center behaves with series of $g$-monomials. Notice that this is necessary information in the classical case too. In fact, this relation is what allows a power series to be analytic on a certain interval.
\begin{pro}
	\label{cambiodepuntosuma}
	Let $g:\mathbb{R}\to\mathbb{R}$ be a derivator and fix some $x_0\in\mathbb{R}$. If
	\[  f(x)=\sum_{n=0}^\infty a_n g_n(x)\]  
	converges absolutely on $[c_1,c_2]$, for some $c_1,c_2\in\mathbb R$ such that $x_0\in(c_1,c_2)$. Then, if $s\in(x_0,c_2 )$,
	\begin{equation}
		\label{1111}
	 f(x)=	\sum_{k=0}^\infty \left(\frac{1}{k!}\sum_{n=k}^\infty a_n\frac{n!}{(n-k)!}g_{n-k}(s)\right) g_{s,k}(x)
	\end{equation}
	on $[s,c_2]$, where the series converges absolutely. If we choose $s\in(c_1,x_0)$, then ~\eqref{1111} holds on $[c_1,s]$, where the series converges absolutely. 
\end{pro}
\begin{proof}
	Take $s\in(x_0,c_2 )$. Note that for all $x\in[s,c_2]$, thanks to Proposition~\ref{relpolprop},
	\[  g_n(x)=\sum_{k=0}^n {n\choose k}g_{s,k}(x)g_{n-k}(s). \]  
	Every addend is positive, hence the sum
	\[  \sum_{n=0}^\infty a_n \sum_{k=0}^n {n\choose k}g_{s,k}(x)g_{n-k}(s) \]  
	is absolutely convergent. Applying Theorem~\ref{sumasij},
	\begin{align*}
		\sum_{n=0}^\infty a_n g_n(x)&=\sum_{n=0}^\infty a_n \sum_{k=0}^n {n\choose k}g_{s,k}(x)g_{n-k}(s)=\sum_{k=0}^\infty g_{s,k}(x)\sum_{n=k}^\infty a_n {n\choose k} g_{n-k}(s)\\
		&=\sum_{k=0}^\infty \left(\frac{1}{k!}\sum_{n=k}^\infty a_n\frac{n!}{(n-k)!}g_{n-k}(s)\right) g_{s,k}(x).
	\end{align*}
	If $s\in(c_1,x_0)$, the proof is identical {having in mind} that, for $x\in[c_1,s]$,
	\[  \left\lvert g_n(x)\right\rvert=\sum_{k=0}^n {n\choose k}\left\lvert g_{s,k}(x)\right\rvert\left\lvert g_{n-k}(s)\right\rvert. \qedhere\]  
\end{proof}

A priori, if we have a series of $g$-monomials that converges on $[x_0,c_2]$, by changing the center point to $s\in(x_0,c_2)$ we can only guarantee convergence on the right side of $s$. But for continuous derivators we can assure convergence on both sides.
\begin{cor}
If $g$ is a continuous derivator and 
\[  f(x)=\sum_{n=0}^\infty a_n g_1(x)^n\]  
converges at some $c\in\mathbb{R}$ such that $g_1(c)\neq0$, then, for all $s\in\mathbb R$, the series~\eqref{1111} converges absolutely on the ball $B_g(s,\left\lvert g_1(c)\right\rvert-\left\lvert g_1(s)\right\rvert)$.
\end{cor}
\begin{proof}
This follows from the proof of Proposition~\ref{cambiodepuntosuma} and that, for all $x\in\mathbb R$,
\[  \sum_{k=0}^n {n\choose k}\left\lvert g_{s,1}(x)\right\rvert^k\left\lvert g_1(s)\right\rvert^{n-k}=(\left\lvert g_{s,1}(x)\right\rvert+\left\lvert g_1(s)\right\rvert)^n.\qedhere\]  
\end{proof}

\subsubsection{Some illustrative examples}

If the derivator is continuous, convergence is not a problem. However, it takes a single discontinuity point to break the convergence. If we have a infinite amount of discontinuities, convergence on the left side gets even more complicated. To illustrate this point, we will present some examples.
%
\begin{exa}
	\label{ejemploII}
	Take $g:\mathbb{R}\to\mathbb{R}$ given by
	\[  
	g(x)= \begin{dcases} 
		x+1, & x>0,\\
		x, & x\leq0.
	\end{dcases}
	\]  
	Let us compute the $g$-monomials centered at $x_0=0$. Following the calculations made in Remark~\ref{contydescguay},
	\[  g^B(x)=\begin{dcases}
		1, & x\in(0,\infty),\\
		0, & x\in(-\infty,0].
	\end{dcases}\]  
	Hence, for $n\geq1$,
	\[  g_n(x)=\begin{dcases}
		x^n+nx^{n-1}, & x\in(0,\infty),\\
		x^n, & x\in(-\infty,0].
	\end{dcases}
	\]  
	If we take
	\[  f(x)=\sum_{n=0}^\infty g_n(x),\]  
	we have that the previous sum converges for $\left\lvert x\right\rvert<1$. Besides,
	\[  f(x)=\begin{dcases}
		\frac{1}{1-x}+\frac{1}{(1-x)^2}, & x\in(0,1),\\
		\frac{1}{1-x}, & x\in(-1,0].
	\end{dcases}\]  
	Fix $s\in(0,1)$, we will calculate the series of $g$-monomials of $f$ centered at $s$. {Having in mind} that
	\[  \frac{k!}{(1-x)^{k+1}}=\sum_{n=k}^\infty \frac{n!}{(n-k)!} x^{n-k},
	\]  
for $x\in(-1,1)$,	let us compute the sequence of coefficients that appears in the formula~\eqref{1111}. We have that
	\begin{align*}
		\frac{1}{k!}\sum_{n=k}^\infty \frac{n!}{(n-k)!}g_{n-k}(s)&=\frac{1}{k!}\left(\sum_{n=k}^\infty \frac{n!}{(n-k)!}s^{n-k}+\sum_{n=k+1}^\infty\frac{n!}{(n-k)!}(n-k)s^{n-k-1}\right)\\
		&=\frac{1}{k!}\left(\sum_{n=k}^\infty \frac{n!}{(n-k)!}s^{n-k}+\sum_{n=k+1}^\infty\frac{n!}{(n-k-1)!}s^{n-k-1}\right)\\
		&=\frac{1}{(1-s)^{k+1}}+\frac{k+1}{(1-s)^{k+2}}.
	\end{align*}
	Now, since $g_{s,1}^B(x)=g^B(x)-1$ and $g_{s,1}^C(x)=x-s$, for $n\geq1$,
	\[  
	g_{s,n}(x)=(x-s)^n, \quad x\in(0,\infty).
	\]  
	Since $g$ is a derivator whose points of discontinuity are isolated, we can guarantee convergence on the left side of $s$. For $x\in(0,1)$ such that $\left\lvert x-s\right\rvert<\left\lvert 1-s\right\rvert$,
	\begin{align*}
		f(x)&=\sum_{k=0}^\infty \left(\frac{1}{k!}\sum_{n=k}^\infty \frac{n!}{(n-k)!}g_{n-k}(s)\right) g_{s,k}(x)\\&=\sum_{k=0}^\infty\left( \frac{1}{(1-s)^{k+1}}+\frac{k+1}{(1-s)^{k+2}}\right)(x-s)^k. 
	\end{align*}
	However, for $x=0$, we lose the convergence of the series. Note that $g_{s,1}^B(x)$ is like in Example~\ref{ejemploeee} for $h=1$ and, therefore, $g_{s,k}^B(x)=(-1)^kk!$ for all $x\in(-\infty,0]$. We know that $\left\lvert g_{s,k}^B(x)\right\rvert\leq\left\lvert g_{s,k}(x)\right\rvert$, for all $x\in\mathbb{R}$. Then,
	\begin{equation}\label{01103201662163032000143000} \left\lvert \frac{1}{(1-s)^{k+1}}+\frac{k+1}{(1-s)^{k+2}}\right\rvert k!\leq\left\lvert \frac{1}{(1-s)^{k+1}}+\frac{k+1}{(1-s)^{k+2}}\right\rvert\left\lvert g_{s,k}(x)\right\rvert, \end{equation}
	for all $x\leq0$. Since the left side of~\eqref{01103201662163032000143000} does not tend to $0$ when $k$ tends to infinity, the sum
	\[  \sum_{k=0}^\infty\left( \frac{1}{(1-s)^{k+1}}+\frac{k+1}{(1-s)^{k+2}}\right) g_{s,k}(x)\]  
	does not converge for $x\leq0$. No matter how close we get to $0$, for all $s\in(0,1)$, we only have convergence on a neighborhood of $s$ contained on $(0,1)$. Note that the size of the jump at $0$ does not matter either.
\end{exa}
\begin{exa}
	\label{ejemploIII}
	Let $\{x_k\}_{k\in\mathbb N}$ be a sequence that converges to $0$ such that $x_k<0$ for all $k\in\mathbb N$. Take a function $\Delta g:\mathbb{R}\to [0,\infty)$ such that $\Delta g(x_k)>0$ for all $k\in\mathbb N$, \[  \sum_{k\in\mathbb N}\Delta g(x_k)<\infty,\]  and $\Delta g(x)=0$ for all $x\in\mathbb{R}-\{x_k\ |\ k\in\mathbb N\}$. Define $g:\mathbb{R}\to\mathbb{R}$ as:
	\[  
	g(x)= \begin{dcases} 
		x, & x>0,\\
		-\sum_{{s\in}[x,0)}\Delta g({s}), & x\leq0.
	\end{dcases}
	\]  
	Fix $x_0=0$. We have that $g_n(x)=x^n$ for all $x\geq0$ and $n\in\mathbb N$. Consider the sum
	\begin{equation}
		\label{49657}
		 f(x)=\sum_{n=0}^\infty g_n(x).
	\end{equation}
	It is clear that for $x\in[0,1)$, the sum converges absolutely. Now, fix some $x<0$. If we take $g^m$ like in Section~\ref{1286}, the derivators $g^m$, on the left side of $0$, are derivators such that $(g^m)^C=0$ and the set of discontinuities is a set of isolated points. Take some $x_k\in[x,0)$. From a certain $m\in\mathbb N$, that $x_k$ will be a discontinuity point of $g^m$. Applying the formula of Proposition~\ref{gmondiscontinua} for points to the left side of $0$,
	\[  \left\lvert g_n^m(x)\right\rvert\geq n!\Delta g(x_k)^n.\]  
	Then
	\[  \left\lvert g_n(x)\right\rvert\geq n!\Delta g(x_k)^n,\]  
	and hence $\left\lvert g_n(x)\right\rvert\to \infty$ when $n$ tends to infinity. In particular,~\eqref{49657} does not converge for any $x<0$. Note that the $g_n$ are left-continuous and hence tend to $0$ as we get closer to $0$. However, fixing any $x<0$, the absolute value tends to infinity when $n$ tends to infinity. 
\end{exa}
\begin{pro}\label{sucesionizq}
Suppose $\{a_n\}_{n\in\mathbb N}\subset\mathbb F$ is such that the series $\sum_{n=0}^{\infty}\left\lvert a_n\right\rvert\left\lvert g_n(c)\right\rvert$ converges at some discontinuity point $c\in(-\infty,x_0)$ of $g$ ($\Delta g(c)\neq0$). Then there exists some constant $M>0$ such that 
\[  \left\lvert a_n\right\rvert\leq \frac{M^{n+1}}{n!}\]  
for all $n\geq0$.
\end{pro}
\begin{proof}
Applying Proposition~\ref{gmondiscontinua}, $\left\lvert g_n(c)\right\rvert\geq n!\Delta g(c)^n$. Hence $\sum_{n=0}^{\infty}\left\lvert a_n\right\rvert n! \Delta g(c)^n$ converges. By the root test, there exists some natural $m$ such that
\[  (\left\lvert a_n\right\rvert n!)^{\frac{1}{n}}\leq \frac{1+\varepsilon}{\Delta g(c)} \Rightarrow \left\lvert a_n\right\rvert\leq \frac{1}{n!}\left(\frac{1+\varepsilon}{\Delta g(c)}\right)^n \]  
for all $n\geq m$, where $\varepsilon>0$ is some positive number. Take \[  M=\max\left\{1,\frac{1+\varepsilon}{\Delta g(c)},\left\lvert a_0\right\rvert0!,\left\lvert a_1\right\rvert1!,\dots,\left\lvert a_{m-1}\right\rvert(m-1)!\right\}.\]  Then,
\[  \left\lvert a_n\right\rvert\leq \frac{M^{n+1}}{ n! }\]  
for all $n\geq0$.
\end{proof}

This means that if a $g$-monomial series manages to absolutely converge {at a} discontinuity point on the left side of $x_0$, then it must absolutely converge on $[x_0,+\infty)$, see Section~\ref{exponencial}.

\subsection{Stieltjes-analyticity}
Since asking for convergence on the left side of $x_0$ is a very strong hypothesis, recall Proposition~\ref{sucesionizq}, we decide to give this definition of Stieltjes-analytic function.
\begin{dfn}
	\label{defanal}
	Let $g:\mathbb{R}\to \mathbb{R}$ a derivator. Given $ f:\Omega\to\mathbb{F}$, and $\Omega\subset\mathbb{R}$ a (usual) open set. We say $ f$ is \emph{Stieltjes-analytic} on $\Omega$, if, for all $y\in\Omega$, there exist $\delta>0$ and $t\in\mathbb{R}$ such that $y\in(t,t+\delta)$ and a sequence $\{a_n\}_{n=0}^\infty\subset\mathbb{F}$ satisfying
	\[  f(x)=\sum_{n=0}^\infty a_n g_{t,n}(x)\]  
	for all $x\in(t,t+\delta)\subset\Omega${, and, moreover, the sum converges absolutely.}
\end{dfn}
With this definition, the function $f$ in Example~\ref{ejemploIII} is Stieltjes-analytic on $\Omega=(0,1)$. For any $y\in \Omega$, taking $t=0$ the definition is satisfied. Note how this differs from the usual definition, for a function to be analytic at any given point it has to be written as a power series centered at that point. This is something we do not ask for here. It is easy to prove that the function $f$ in Example~\ref{ejemploII} is Stieltjes-analytic on $\Omega=(-1,1)$, since in that case $g$ is a derivator whose set of discontinuities is a set of isolated points. \par Definition~\ref{defanal} is equivalent to the usual when $g=\operatorname{Id}$. The space of Stieltjes-analytic functions defined on the same domain is a vector space. We will see how Stieltjes-analyticity behaves with integration and derivation. To do so, we will present some results that relate these concepts to the series of $g$-monomials.
\begin{pro}
	\label{gcontsec}
	Let $\{f_n\}_{n\in\mathbb N}$ be a sequence of functions defined on $X\subset \mathbb{R}$ that converges uniformly to some function $f$. If $f_n$ is $g$-continuous for all $n\in\mathbb N$, then $f$ is $g$-continuous as well.
\end{pro}
The proof of this statement is practically identical to that of \cite[Theorem 3.4]{RodrigoTheory}. As a consequence, we have that $\mathcal{BC}_g(X,\mathbb R)$ is a Banach space. It is easy to see now that any Stieltjes-analytic function is $g$-continuous in every point of their domain. In particular, any Stieltjes-analytic function is $g$-continuous.

\subsubsection{Stieltjes-analyticity and integration}

Lets see how series of $g$-monomials behave with integration. Thanks to the Dominated Convergence Theorem, we will prove absolute convergence of the series of integrals on the same set where the $g$-monomial series converges.

\begin{pro}
	\label{seriesint}
	Let $g:\mathbb{R}\to\mathbb{R}$ be a derivator and fix some $x_0\in\mathbb{R}$. Suppose that 
	\[  f(x)=\sum_{n=0}^\infty a_n g_n(x)\]  
	converges absolutely on $[c_1,c_2]$, for some $c_1,c_2\in\mathbb R$ such that $x_0\in(c_1,c_2)$. Then $f$ is $g$-integrable in $[c_1,c_2]$ and
	\[  \int_{x_0}^xf\operatorname{d}\mu=\sum_{n=0}^\infty \frac{a_n}{n+1} g_{n+1}(x)\]  
	for all $x\in[c_1,c_2]$.
	 Besides, the previous sum converges uniformly and absolutely on $[c_1,c_2]$.
\end{pro}
\begin{proof}
	Thanks to Proposition~\ref{gcontsec}, $f$ is $g$-continuous in $[c_1,c_2]$. Besides,
	\[  \left\lvert f(x)\right\rvert\leq \max_{i=1,2}\sum_{n=0}^\infty \left\lvert a_n\right\rvert\left\lvert g_n(c_i)\right\rvert.\]  
	Then $f\in\mathcal{BC}_g([c_1,c_2])$ and hence $f$ is $g$-integrable. Applying Dominated Convergence Theorem, for all $x\in[c_1,c_2]$,
	\begin{equation}
		\label{3485}
		\int_{x_0}^xf\operatorname{d}\mu=\lim_{m\to\infty}\sum_{n=0}^m a_n\int_{x_0}^x g_n\operatorname{d}\mu_g=\lim_{m\to\infty}\sum_{n=0}^m \frac{a_n}{n+1} g_{n+1}(x)=\sum_{n=0}^\infty \frac{a_n}{n+1} g_{n+1}(x).
	\end{equation}
	Denote $h$ as
	\[  h(x)=\sum_{n=0}^\infty \left\lvert a_n\right\rvert\left\lvert g_n(x)\right\rvert,\]  
	for $x\in[c_1,c_2]$. We have that $h$ is $g$-integrable since it is $g$-measurable and bounded. We can apply the Dominated Convergence Theorem to $h$ again so 
{
\begin{align*}
\sum_{n=0}^\infty \left\lvert \frac{a_n}{n+1}\right\rvert\left\lvert g_{n+1}(x)\right\rvert=\sum_{n=0}^\infty \left\lvert a_n\right\rvert \int_{[x_0,x)}\left\lvert g_n\right\rvert\operatorname{d}\mu_g= \int_{[x_0,x)} h\operatorname{d}\mu_g,\\
\sum_{n=0}^\infty \left\lvert \frac{a_n}{n+1}\right\rvert\left\lvert g_{n+1}(x)\right\rvert=\sum_{n=0}^\infty \left\lvert a_n\right\rvert \int_{[x,x_0)}\left\lvert g_n\right\rvert\operatorname{d}\mu_g= \int_{[x,x_0)} h\operatorname{d}\mu_g.
\end{align*}
The series~\eqref{3485} is absolutely convergent for $x\in[c_1,c_2]$.
}{We get uniform convergence on $[c_1,c_2]$ from Proposition~\ref{convergenciags}}. 
\end{proof}

\subsubsection{Stieltjes-analyticity and {differentiability}}

Although integrability behaves well with series of $g$-monomials, we have to ask for more assumptions to ensure the convergence of the series of derivatives. The derivative will exist at points where $\left\lvert g_1\right\rvert$ is strictly less than the extremes of the convergence interval. Note that this is also required in the usual case. However, here it leads to more problems due to the constancy intervals of $g$ --see Example~\ref{01324856}.
%
\begin{pro}
	\label{derseries}
	Let $g:\mathbb{R}\to\mathbb{R}$ be a derivator and fix some $x_0\in\mathbb{R}$. Suppose that
	\[  f(x)=\sum_{n=0}^\infty a_n g_n(x)\]  
	converges absolutely on $[c_1,c_2]$, for some $c_1,c_2\in\mathbb R$ such that $x_0\in(c_1,c_2)$. If $c\in(x_0,c_2)$ is such that $\left\lvert g_1(c)\right\rvert<\left\lvert g_1(c_2)\right\rvert$, the series
	\[  \sum_{n=1}^\infty n{a_n} g_{n-1}(x)\]  
	converges absolutely for $x\in[x_0,c]$. Besides, for any $x\in(x_0,c)-C_g$, 
	\begin{equation}\label{etiqueta1} f'_g(x)=\sum_{n=1}^\infty n{a_n} g_{n-1}(x).\end{equation}
	If{~\eqref{2222221354324} holds with $\Omega=(x_0,c)$} then~\eqref{etiqueta1} holds for all $x\in(x_0,c)$. Analogously, the same applies if ${c\in(c_1,x_0)}$ and $\left\lvert g_1(c)\right\rvert<\left\lvert g_1(c_1)\right\rvert$.
\end{pro}
\begin{proof}
	Take $c\in(x_0,c_2)$ such that $\left\lvert g_1(c)\right\rvert<\left\lvert g_1(c_2)\right\rvert$. If we follow the proof of Proposition~\ref{cambiodepuntosuma}, we have that the series
	\begin{equation}\label{ssdlolk} \sum_{k=0}^\infty \left(\frac{1}{k!}\sum_{n=k}^\infty\left\lvert a_n\right\rvert\frac{n!}{(n-k)!}\left\lvert g_{n-k}(c)\right\rvert\right) \left\lvert g_{c,k}(c_2)\right\rvert\end{equation}
	converges. Therefore, since 
	\[  \left\lvert g_1(c)\right\rvert<\left\lvert g_1(c_2)\right\rvert\Leftrightarrow\left\lvert g_{c,1}(c_2)\right\rvert\neq0\]  
	and the coefficient associated to $\left\lvert g_{c,1}(c_2)\right\rvert$ in the series~\eqref{ssdlolk} is
	\begin{equation}\label{1544987622254125}\sum_{n=1}^\infty n\left\lvert a_n\right\rvert\left\lvert g_{n-1}(c)\right\rvert,\end{equation}
	the sum~\eqref{1544987622254125} must converge. In particular, we have that the series
	\[  \sum_{n=1}^\infty n{a_n} g_{n-1}(x)\]  
	converges absolutely on $[x_0,c]$. Applying Proposition~\ref{seriesint} to the previous sum, if $x\in[x_0,c]$,
	\[  f(x)-f(0)=\int_{[x_0,x)}\sum_{n=1}^\infty n{a_n} g_{n-1}(s)\operatorname{d}\mu_g(s).\]  
	Thanks to Proposition~\ref{primder}, 
	\[  f'_g(x)=\sum_{n=1}^\infty n{a_n} g_{n-1}(x)\]  
	for all $x\in(x_0,c)-C_g$. {If~\eqref{2222221354324} holds with $\Omega = (x_0,c)$}, the $g$-derivative is well defined at all points, so equality~\eqref{etiqueta1} holds for all $x\in(x_0,c)$. The same arguments are valid for points to the left of $x_0$. 
\end{proof}

\begin{thm}
	\label{1155}
	Let $g:\mathbb{R}\to\mathbb{R}$ be a derivator. If $f:\Omega\to\mathbb{F}$ is a Stieltjes-analytic function that satisfies:
{\rm
	\begin{enumerate}[noitemsep, itemsep=.1cm]
		\item[({$H$})]For all $y\in\Omega$, there exist $0<\delta$ and $t\in\mathbb{R}$ such that $y\in(t,t+\delta)$, $g^C(y)<g^C(t+\delta)$ and a sequence $\{a_n\}_{n=0}^\infty\subset\mathbb{F}$ satisfying
		\[  f(x)=\sum_{n=0}^\infty a_n g_{t,n}(x)\]  
		for all $x\in(t,t+\delta)\subset\Omega${, and, moreover, the sum converges absolutely.}

	\end{enumerate}}
Then $\Omega$ satisfies condition~\eqref{2222221354324}, $f$ is $g$-differentiable on $\Omega$ and ${f'_g:\Omega\to\mathbb{F}}$ is again a Stieltjes-analytic function that satisfies $({H})$.
\end{thm}
\begin{proof}
	Take $y\in\Omega$. By $({H})$, there are $0<\delta$, $t\in\mathbb{R}$ and a sequence $\{a_n\}_{n=0}^\infty\subset\mathbb{F}$ such that $y\in(t,t+\delta)$ and
	\[  f(x)=\sum_{n=0}^\infty a_n g_{t,n}(x)\]  
	for all $x\in(t,t+\delta]\subset\Omega$ and $g^C(y)<g^C(t+\delta)$ (we can assume that the series converges uniformly on $(t,t+\delta]$ choosing a suitable $\delta$). By the continuity of $g^C$, there exists $\delta'\in(0,\delta)$ such that $y\in(t,t+\delta')$ and $g^C(y)<g^C(t+\delta')<g^C(t+\delta)$. Applying Proposition~\ref{derseries}, the series
	\begin{equation} 
		\label{7756}
		\sum_{n=1}^\infty na_n g_{t,n-1}(x)
	\end{equation}
	converges absolutely on $[t,t+\delta']$. Besides, we know that $f'_g(x)$ exists at $x\in(t,t+\delta')-C_g$ and equals the series~\eqref{7756}. We can repeat the previous argument until we reach the following conclusion, for \[  t'=\sup\{x\in\mathbb{R}:g^C(x)<g^C(t+\delta)\}, \]  the series~\eqref{7756} converges absolutely on $[t,t')$ and equals $f'_g$ on $(t,t')-C_g$. Note that, if $x\in(t,t')$, then $g^C(x)<g^C(t')$. Suppose $x\in(t,t')\cap C_g$ and $x\in (a_n,b_n)\subset C_g$. We have that $g^C(x)=g^C(b_n)<g^C(t')$ hence $[x,b_n]\subset(t,t')\subset\Omega$ since $b_n<t'$. Note that this proves that both $(t,t')$ and $\Omega$ satisfy hypothesis~\eqref{2222221354324}. Then, the series~\eqref{7756} equals $f'_g$ on $(t,t')$ by Definition~\ref{derivadadef}. Choosing adequate $t$ and $\delta$ we have that $f'_g$ satisfies $({H})$ at $y$. Since $y$ was arbitrarily chosen, we have that $f'_g$ is Stieltjes-analytic.
\end{proof}
Note that the hypotheses made in Theorem~\ref{1155} are automatically fulfilled if $g$ is continuous and strictly increasing. Take $a,b\in\mathbb{R}$ such that $a<b$, $a\notin N^{-}_g$, $b\notin D_g\cup N^{+}_g\cup C_g$ and $[a,b]\subset\Omega$, if $f$ is a Stieltjes-analytic function that satisfies $(1)$, applying Theorem~\ref{1155} recursively we obtain that 
\[  f\in\mathcal{C}_g^\infty([a,b],\mathbb{F}).\]  
In any case, we have
\[  f\in\mathcal{C}_g^\infty(\Omega,\mathbb{F}).\]  
\begin{exa}
	\label{01324856}
	Take $g:\mathbb{R}\to\mathbb{R}$ given by
	\[  
	g(x)= \begin{dcases} 
		x, & x\leq 0, \\
		n, & x\in(n-1,n],\quad n\in\mathbb N.
	\end{dcases}
	\]  
%
%
%
	Fix $x_0=0$. It is clear that $g_n(x)=x^n$ for $x\in(-\infty,0]$. Now, $g_2(x)=0$ for $x\in(0,1]$, $g_3(x)=0$ for $x\in(0,2]$ and, in general, $g_n(x)=0$ for $x\in(0,n-1]$ for all $n\geq2$. In particular, any $g$-monomial series is finite for all $x\in(0,\infty)$, since for $n\in\mathbb N$ greater than some fixed natural, $g_n(x)=0$. Consider the series
	\[  f(x)=\sum_{n=0}^\infty n!g_n(x).\]  
	The series does not converge at points $x<0$, since the sequence $\{n!x^n\}_{n=0}^\infty$ does not tend to $0$ when $n$ tends to infinity. Even so, it converges absolutely at $[0,\infty)$, since for every point $x>0$ the sum is finite. In particular, we have that $f$ satisfies Definition~\ref{defanal} in $(0,\infty)$. In fact, $f$ has $g$-derivatives of all orders that are Stieltjes-analytic functions too. Consider now $f$ defined only in $(0,y)$, for some $y>0$. Take $n\in\mathbb N$ such that $n-1<y\leq n$. Since $g_{n+1}(x)=0$ for $x\in(0,n]$, the coefficients $\{k!\}_{k=0}^\infty$ that define $f$ are not unique anymore. We are free to choose the coefficients for $k\geq n+1$. 

	Note that we are in a similar situation to that of Example~\ref{55886}. Since $f$ is defined only on $(0,y)$, we can only differentiate at points in $(0,n-1]$. Its derivative can only be differentiated at points in $(0,n-2]$, and so on. The coefficients will tell us what values $f'_g$ takes on $(n-1,y]$, $f^{(2)}_g$ on $(n-2,y]$, etc.
\end{exa}
The last example proves that the uniqueness of the coefficients of a Stieltjes-analytic function is not ensured. Combining the arguments made in Theorem~\ref{1155} and Proposition~\ref{derseries} we obtain the following result.
\begin{pro}
	\label{derseries2}
	Let $g:\mathbb{R}\to\mathbb{R}$ be a derivator and fix some $x_0\in\mathbb{R}$. If
	\[  f(x)=\sum_{n=0}^\infty a_n g_n(x)\]  
	converges absolutely on $[c_1,c_2]$, for some $c_1,c_2\in\mathbb R$ such that $x_0\in(c_1,c_2)$, $g_1^C(c_2)>0$ and $g_1(c_1)<0$. If
	\[  t_2=\sup\{x\in\mathbb{R}:g_1^C(x)<g_1^C(c_2)\}\]  
	and
	\[  t_1=\inf\{x\in\mathbb{R}:g_1(c_1)<g_1(x)\},\]  
	the series
	\[  \sum_{n=k}^\infty a_n\frac{n!}{(n-k)!} g_{n-k}(x)\]  
	converges absolutely on $(t_1,t_2)$, for all $k\in\mathbb N$. Besides, $(t_1,t_2)$ satisfies~\eqref{2222221354324} and, for all ${x\in(t_1,t_2)}$, 
	\[  f^{(k)}_g(x)=\sum_{n=k}^\infty a_n\frac{n!}{(n-k)!} g_{n-k}(x).\]  
\end{pro}
\begin{proof}
	Note that $x_0\in(t_1,t_2)$. Besides, if $x\in(t_1,t_2)$, then $g_1^C(x)<g_1^C(t_2)$ and $g_1(t_1)<g_1(x)$. That is because $g_1^C(t_2)=g_1^C(c_2)$ and $g_1(t_1)=g_1(c_1)$. Take $x\in(x_0,t_2)$. We have that $x<c_2$. Hence, from Proposition~\ref{cambiodepuntosuma}, the series
	\[  
	\sum_{k=0}^\infty \left(\frac{1}{k!}\sum_{n=k}^\infty\left\lvert a_n\right\rvert\frac{n!}{(n-k)!}\left\lvert g_{n-k}(x)\right\rvert\right) \left\lvert g_{x,k}(c_2)\right\rvert
	\]  
	converges. Now, since $g_1^C(x)<g_1^C(c_2)$, $g_{x,k}(c_2)\neq0$ for all $k\in\mathbb N$, so the series
	\[  
	\sum_{n=k}^\infty\left\lvert a_n\right\rvert\frac{n!}{(n-k)!}\left\lvert g_{n-k}(y)\right\rvert
	\]  
	converges for all $k\in\mathbb N$ and $y\in[x_0,x]$. Applying Proposition~\ref{seriesint} and Proposition~\ref{primder} recursively we have that
	\begin{equation}
		\label{444444}
		f^{(k)}_g(y)=\sum_{n=k}^\infty a_n\frac{n!}{(n-k)!}g_{n-k}(y)\end{equation}
	for $y\in(x_0,x)-C_g$. With $x$ approaching $t_2$, the series~\eqref{444444} converges absolutely on $[x_0,t_2)$ and equality~\eqref{444444} holds for all $y\in(x_0,t_2)$.
	To the left of $x_0$, the argument is the same. If $x\in(t_1,x_0)$ we only need that $g_1(c_1)<g_1(x)$ to ensure $g_{x,k}(c_1)\neq0$ for all $k\in\mathbb N$. Note that, if the hypotheses are satisfied at both sides, we can $g$-differentiate at $x_0\in(t_1,t_2)$.
\end{proof}

Note that if $g=\operatorname{Id}$, the last result guarantees that, if $(-R,R)$ is the convergence interval of a power series, it is the convergence interval of the series of its derivatives as well. We have achieved this without applying the Cauchy–Hadamard theorem \cite[Lemma 1.1.6]{Krantz1992}.

\subsubsection{Coefficients of a Stieltjes-analytic function}

We have seen that the relationship between a Stieltjes-analytic function and its coefficients can be more complex than it seems, see Example~\ref{01324856}. Even so, that relation behaves like in the classical case if we ask for the right hypotheses.
\begin{thm}
	\label{2216547855}
	Let $f$ be a Stieltjes-analytic function defined on $\Omega$ {and fix $y\in\Omega$. Assume $f$ satisfies $(H)$, that is, there are $\delta>0$, $t\in\mathbb{R}$ and a sequence $\{a_n\}_{n=0}^\infty\subset\mathbb{F}$ such that $y\in(t,t+\delta)$,
	\[  f(x)=\sum_{n=0}^\infty a_n g_{t,n}(x)\]  
	for all $x\in(t,t+\delta)\subset\Omega$ and $g^C(y)<g^C(t+\delta)$ with the series converging absolutely}. Suppose that $t\in\Omega$. Then
	\[  \frac{f^{(n)}_g(t)}{n!}=a_n\]  
	for $n\geq0$.
\end{thm}
\begin{proof}
	From Theorem~\ref{1155}, we have $g$-derivatives of $f$ of all orders at $t\in\Omega$. From Proposition~\ref{derseries2}, exists $t'\in(y,t+\delta)$ such that
	\begin{equation}
		\label{44444444}
		f^{(k)}_g(x)=\sum_{n=k}^\infty a_n\frac{n!}{(n-k)!}g_{t,n-k}(x)
	\end{equation}
	for all $x\in(t,t')$ and $k\in\mathbb N$. Take $t''\in(t,t')$ to ensure uniform convergence of the series~\eqref{44444444} on $[t,t'']$. Then,
	\begin{align*}
		\lim_{x\to t^+}f^{(k)}_g(x)&=\lim_{m\to\infty}\sum_{n=k}^m \lim_{x\to t^+}a_n\frac{n!}{(n-k)!}g_{t,n-k}(x)\\
		&=a_kk!+a_{k+1}(k+1)!\Delta g(t),
	\end{align*}
	for all $k\in\mathbb N$. Now, if $\Delta g(t)=0$, $f^{(k)}_g$ is continuous at $t$ and
	\[  f^{(k)}_g(t)=\lim_{x\to t^+}f^{(k)}_g(x)=a_kk!.\]  
	Hence, \[  \frac{f^{(k)}_g(t)}{k!}=a_k\]  for all $k\geq0$. If $\Delta g(t)\neq0$, $t$ is a discontinuity point of $g$ and then, for $k\geq1$,
	\[  
	f^{(k)}_g(t)=\lim_{x\to t^+}\frac{f^{(k-1)}_g(x)-f^{(k-1)}_g(t)}{g(x)-g(t)}=a_{k}k!.
	\]  
	Hence,
	\[  \frac{f^{(k)}_g(t)}{k!}=a_k\]  
	for all $k\geq0$.
\end{proof}
\begin{rem}
	Note that, since $f$ satisfies the hypotheses of Theorem~\ref{1155}, $f\in\mathcal{C}_g^\infty(\Omega,\mathbb{F})$ and $f$ allows $g$-derivatives of all orders on $C_g$. If $t\in C_g$, then $\Delta g(t)=0$ and the argument we gave in Theorem~\ref{2216547855} holds.
\end{rem}

\section{Differential equations and applications}
\label{exponencial}
\subsection{Differential equations}

Assume from now on that $\infty\notin N_g^+$, just to guarantee intervals not bounded from above satisfy condition~\eqref{2222221354324}. Many of the things we say here are true without that assumption.
\par The concept of Stieltjes-analytic function was born with the aim of solving differential equations, especially, linear differential equations. We now give a method to solve any higher order linear homogeneous Stieltjes differential equations with constant coefficients that works for some nonhomogeneous cases as well. Fix some $x_0\in\mathbb{R}$, consider the initial value problem 
\begin{equation}
	\label{01111}
	\begin{dcases}
		v^{(m)}_g(x)=\sum_{k=0}^{m-1}\lambda_k v^{(k)}_g(x), & \lambda_k\in\mathbb{F},\\
		v^{(k)}_g(x_0)=c_k,& c_k\in\mathbb{F}, \ k\in\{0,\dots,m-1\}.
	\end{dcases}
\end{equation}
We can study whether there is a Stieltjes-analytic solution to this problem. Suppose that a Stieltjes-analytic solution exists and we can center its $g$-monomial series on $x_0$. Then,
\[  v(x)=\sum_{n=0}^{\infty}a_ng_n(x)\]  
for some coefficients $\{a_n\}_{n\in\mathbb N}\subset\mathbb R$. From problem~\eqref{01111} and
\[  v^{(k)}_g(x)=\sum_{n=k}^{\infty}a_n\frac{n!}{(n-k)!}g_{n-k}(x)\]  
we obtain the following difference equation by matching the coefficients:
\begin{equation}
	\label{011111}
	\begin{dcases}
		a_{n+m}(n+m)!=\sum_{k=0}^{m-1}\lambda_k a_{n+k}(n+k)!, & \lambda_k\in\mathbb{F}, \ n\geq0,\\
		a_kk!=c_k,& c_k\in\mathbb{F}, \ k\in\{0,\dots,m-1\}.
	\end{dcases}
\end{equation}

\begin{thm}
	\label{6665444548132184}
	Let $\lambda_0, \lambda_1, \dots, \lambda_{m-1} \in \mathbb{F}$. Given the $m^{\text{th}}$-order linear difference equation with constant coefficients 
	\begin{equation}
		\label{0111116543216854}
		a_{n+m}+ \lambda_{m-1} a_{n+m-1} +\dots + \lambda_{0} a_n =0,\quad n\geq0.\end{equation}
	If $c_0, c_1, \dots, c_{m-1}$ are real numbers, there is a unique solution of~\eqref{0111116543216854}, that satisfies
	\[  a_0=c_0, \,\,a_1 = c_1, \,\,\dots, \,\,a_{m-1}=c_{m-1}.\]  

\end{thm}
{This} result can be found in \cite[Theorem 4.3]{DE}, along with the explicit solutions of~\eqref{0111116543216854}. From Theorem~\ref{6665444548132184}, there is a unique sequence $\{a_n\}_{n=0}^\infty$ that solves problem~\eqref{011111}. Consider then the $g$-monomial series defined by $\{a_n\}_{n=0}^\infty$. As we will prove in Section~\ref{higherorder}, that function is well defined and actually solves problem~\eqref{01111}. In particular, we have that the problem~\eqref{01111} admits Stieltjes-analytic solutions.

We will for now focus on the first order linear problem as it will help us later to solve higher order equations. Consider the problem
\begin{equation}
	\label{01111111}
	\begin{dcases}
		v'_g(x)=\lambda v(x), \\
		v_g(x_0)=1.
	\end{dcases}
\end{equation}
The associated difference equation would be
\[  
\begin{dcases}
	a_{n+1}(n+1)!=\lambda a_nn!, & n\geq1,\\
	a_0=1.
\end{dcases}
\]  
We then obtain the sequence
\[  a_n=\frac{\lambda^n}{n!}.\]  
In the classical case we obtain the exponential function when we consider the power series. The exponential is clearly a solution of~\eqref{01111111} when $g=\operatorname{Id}$. We will see that this actually translates to the general case.
\subsection{The exponential series}
\label{exposeries}
Due to how the literature understands the concept of exponential function associated with a derivator \cite{RodrigoTheory,Onfirst}, we will talk about the exponential series instead of the exponential function. Later, we will see how the two concepts are related. In any case, we will prove that the exponential series in general solves the differential equation~\eqref{01111111}, as well as certain properties that are deduced from the series of $g$-monomials. Let $g:\mathbb{R}\to \mathbb{R}$ be a derivator, $\lambda\in\mathbb{F}-\{0\}$ and fix some $x_0\in \mathbb{R}$. Consider the series
\[  
\sum_{n=0}^\infty \lambda^n \frac{g_n(x)}{n!}.
\]  
Take $x>x_0$. We have that
\[  
\sum_{n=0}^\infty \left\lvert \lambda\right\rvert^n \frac{g_n(x)}{n!}\leq \sum_{n=0}^\infty \left\lvert \lambda\right\rvert^n \frac{g_1(x)^n}{n!}\leq e^{\left\lvert \lambda\right\rvert g_1(x)}.
\]  
Hence, the series converges absolutely on $[x_0,\infty)$. Thus, the previous series defines a Stieltjes-analytic function on the set $(x_0,\infty)$. Take now $x<x_0$ such that $x\in B_g(x_0,\left\lvert \lambda\right\rvert^{-1})$. Since $g$ is left-continuous, the ball $B_g(x_0,\left\lvert \lambda\right\rvert^{-1})$ contains a neighborhood $(x_0-\delta,x_0]$ for some $\delta>0$. We have that

\begin{equation}
	\label{0003215445987}
	\sum_{n=0}^\infty \left\lvert \lambda\right\rvert^n \left\lvert \frac{g_n(x)}{n!}\right\rvert\leq \sum_{n=0}^\infty \left\lvert \lambda\right\rvert^n \left\lvert g_1(x)\right\rvert^n<\infty,
\end{equation}
since $\left\lvert \lambda g_1(x)\right\rvert<1$. We have assured absolute convergence on $B_g(x_0,\left\lvert \lambda\right\rvert^{-1})\cup[x_0,\infty)$ and, hence, on some neighborhood of $x_0$.

\begin{dfn}
	Given $s\in\mathbb{R}$, define \emph{the exponential series associated to $g$ centered at $s$} as the function given by the series
	\[  
	\operatorname{exp}_g(\lambda;s)(x)=\sum_{n=0}^\infty \lambda^n \frac{g_{s,n}(x)}{n!}
	\]  
	on those points where the sum converges absolutely.
\end{dfn}

Choose some $s\in B_g(x_0,\left\lvert \lambda\right\rvert^{-1})$ such that $s<x_0$. Take the exponential series associated to $g$ centered at $s$. Note first that the series
\begin{equation}
	\label{13756}
	\sum_{n=0}^\infty\left\lvert \lambda\right\rvert^n \left\lvert \frac{g_{s,n}(x)}{n!}\right\rvert
\end{equation}
converges absolutely for $x\in B_g(s,\left\lvert \lambda\right\rvert^{-1})\cup[s,\infty)$, applying the same argument we did for $x_0$. Thanks to the product formula for absolutely convergent series \cite[Theorem 3.50]{Rudin}, for some fixed $s$ and all $x\in B_g(s,\left\lvert \lambda\right\rvert^{-1})\cup[s,\infty)$, we have that 
\begin{equation*}
	\begin{aligned}
		\left(\sum_{n=0}^\infty \left\lvert \lambda\right\rvert^{n}\left\lvert \frac{g_{n}(s)}{n!}\right\rvert\right)\left(\sum_{n=0}^\infty\left\lvert \lambda\right\rvert^n \left\lvert \frac{g_{s,n}(x)}{n!}\right\rvert \right)&=\sum_{n=0}^\infty \sum_{k=0}^n \left\lvert \lambda\right\rvert^{k}\left\lvert \frac{g_{k}(s)}{k!}\right\rvert \left\lvert \lambda\right\rvert^{n-k}\left\lvert \frac{g_{s,n-k}(x)}{(n-k)!}\right\rvert\\
		&=\sum_{n=0}^\infty \left\lvert \lambda\right\rvert^{n}\sum_{k=0}^n \left\lvert \frac{g_{k}(s)}{k!}\right\rvert \left\lvert \frac{g_{s,n-k}(x)}{(n-k)!}\right\rvert.
	\end{aligned}
\end{equation*}
Since, applying the change of center formula of Proposition~\ref{relpolprop}, we have that
\[  
\sum_{n=0}^\infty \left\lvert \lambda\right\rvert^n \left\lvert \frac{g_n(x)}{n!}\right\rvert\leq\sum_{n=0}^\infty \left\lvert \lambda\right\rvert^{n}\sum_{k=0}^n \left\lvert \frac{g_{k}(s)}{k!}\right\rvert \left\lvert \frac{g_{s,n-k}(x)}{(n-k)!}\right\rvert,
\]  
the series
\[  
\sum_{n=0}^\infty \lambda^n \frac{g_n(x)}{n!}
\]  
converges absolutely on $B_g(s,\left\lvert \lambda\right\rvert^{-1})\cup[s,\infty)$, which is a bigger set than the one we calculated on~\eqref{0003215445987}. We could try to prove that $\operatorname{exp}_g(\lambda;x_0)$ is defined in the whole real line taking balls of radius $\left\lvert \lambda\right\rvert^{-1}$ recursively. However, if there is a discontinuity with a jump bigger than $\left\lvert \lambda\right\rvert^{-1}$ this process is invalid. Take any $x_0\in\mathbb{R}$, we call $\Omega_{x_0}$ the maximal interval where 
\[  
\sum_{n=0}^\infty \lambda^n \frac{g_{n}(x)}{n!}
\]  
converges absolutely. We have the following result.

\begin{pro}
	Let $g:\mathbb{R}\to \mathbb{R}$ be a derivator and fix some $x_0\in\mathbb{R}$. If $\Delta g(x)<\left\lvert \lambda\right\rvert^{-1}$ for all $x<x_0$, then $\Omega_{x_0}=\mathbb{R}$.
\end{pro}
\begin{proof}
	Clearly, if $\Omega_{x_0}$ is not bounded from below, then $\Omega_{x_0}=\mathbb{R}$. Suppose otherwise and take \[  t=\inf \,\Omega_{x_0}\in\mathbb{R}.\]  Since \[  \lim_{x\to t^+}g(x)-g(t)=\Delta g(t)<\left\lvert \lambda\right\rvert^{-1},\]  there exists $s>t$ such that $\left\lvert g(s)-g(t)\right\rvert<\left\lvert \lambda\right\rvert^{-1}$. Since $s\in\Omega_{x_0}$, we can repeat the calculations made in~\eqref{13756} and deduce that
	\[  
	\sum_{n=0}^\infty \lambda^n \frac{g_n(x)}{n!}
	\]  
	converges absolutely on $B_g(s,\left\lvert \lambda\right\rvert^{-1})$. Now, we can choose $t'<t$ in such a way that $\left\lvert g(s)-g(t')\right\rvert<\left\lvert \lambda\right\rvert^{-1}$ and, hence, $[t',s]\subset B_g(s,\left\lvert \lambda\right\rvert^{-1})\subset \Omega_{x_0}$, which contradicts that $t$ is the infimum of $\Omega_{x_0}$.
\end{proof}
\begin{cor}
\label{omegas}
If $\Omega_{x_0}$ is bounded from below, then $t=\inf \,\Omega_{x_0}\in\mathbb{R}$ is such that $\Delta g(t)\geq \left\lvert \lambda\right\rvert^{-1}$, which means $\Omega_{x_0}=(t,+\infty)$. Besides,
\[  t=\sup\{s\in D_g\cap(-\infty,x_0)\,\left\lvert \, \Delta g(s)\geq\right\rvert\lambda|^{-1}\}.\]  

\end{cor}
\begin{pro}
\label{expcambiopunto}
	For all $t\in\mathbb{R}$, $s\in\Omega_t$ and $x\in\Omega_s$ we have that $x\in\Omega_t$ and
	\[  
	\operatorname{exp}_g(\lambda;t)(s)\operatorname{exp}_g(\lambda;s)(x)=\operatorname{exp}_g(\lambda;t)(x).
	\]  
\end{pro}
\begin{proof}\belowdisplayskip=-12pt
	Repeating the calculations of~\eqref{13756} we have that $x\in\Omega_t$. Again, from the product formula for absolutely convergent series and Proposition~\ref{relpolprop}, we have that
	\begin{equation*}
		\begin{aligned}
			\operatorname{exp}_g(\lambda;t)(s)\operatorname{exp}_g(\lambda;s)(x)&=\sum_{n=0}^\infty \lambda^{n}\frac{g_{t,n}(s)}{n!}\sum_{n=0}^\infty\lambda^n \frac{g_{s,n}(x)}{n!}\\
			&=\sum_{n=0}^\infty \sum_{k=0}^n \lambda^{n-k}\frac{g_{t,n-k}(s)}{(n-k)!}\lambda^k \frac{g_{s,k}(x)}{k!} =\sum_{n=0}^\infty\lambda^n \frac{g_{t,n}(x)}{n!}\\
			&=\operatorname{exp}_g(\lambda;t)(x).
		\end{aligned}
	\end{equation*}
\end{proof}
Applying the last result, it can be proven that $\operatorname{exp}_g(\lambda;t)$ is Stieltjes-analytic on $\Omega_{t}$. Translating Proposition~\ref{expcambiopunto} to the classical case,
\[  \operatorname{exp}(s-t)\operatorname{exp}(x-s)=\operatorname{exp}(t-x).\]  

\begin{pro}
	\label{expprodgcgb}
	Let $g:\mathbb{R}\to \mathbb{R}$ be a derivator and fix some $x_0\in\mathbb{R}$. The series
	\begin{equation}
		\label{63548}
		\sum_{n=0}^\infty \lambda^n \frac{g_n(x)}{n!}
	\end{equation}
	converges absolutely if and only if the series
	\begin{equation}
		\label{32154}
		\sum_{n=0}^\infty \lambda^n \frac{g_n^{C}(x)}{n!},\quad\sum_{n=0}^\infty \lambda^n \frac{g_n^{B}(x)}{n!}
	\end{equation}
	converge absolutely. Besides, for all $x\in\Omega_{x_0}$,
	\begin{equation}
		\label{32154000003265}
		\operatorname{exp}_g(\lambda;x_0)(x)=\operatorname{exp}_{g^C}(\lambda;x_0)(x)\operatorname{exp}_{g^B}(\lambda;x_0)(x).\end{equation}
\end{pro}
\begin{proof}
	If the series~\eqref{63548} converges absolutely, then those in~\eqref{32154} do as well, since $\left\lvert g^{\star}_n(x)\right\rvert\leq\left\lvert g_n(x)\right\rvert$ for all $n\in \mathbb N$, $x\in\mathbb{R}$ and $\star\in\{B,C\}$. Suppose the series~\eqref{32154} converge absolutely, then
	\begin{align}
		\nonumber\sum_{n=0}^\infty \left\lvert \lambda\right\rvert^n \left\lvert \frac{g_n^{C}(x)}{n!}\right\rvert\sum_{n=0}^\infty \left\lvert \lambda\right\rvert^n \left\lvert \frac{g_n^{B}(x)}{n!}\right\rvert&=\sum_{n=0}^\infty\sum_{k=0}^n \left\lvert \lambda\right\rvert^k \left\lvert \frac{g_k^{C}(x)}{k!}\right\rvert \left\lvert \lambda\right\rvert^{n-k} \left\lvert \frac{g_{n-k}^{B}(x)}{(n-k)!}\right\rvert \\
		\nonumber&=\sum_{n=0}^\infty\left\lvert \lambda\right\rvert^n\sum_{k=0}^n \left\lvert \frac{g_k^{C}(x)}{k!}\frac{g_{n-k}^{B}(x)}{(n-k)!}\right\rvert=\sum_{n=0}^\infty\left\lvert \lambda\right\rvert^n \left\lvert \frac{g_n(x)}{n!}\right\rvert.
	\end{align}
	In particular,~\eqref{63548} converges absolutely. Repeating the same calculation without the absolute values we have that
	\[  \operatorname{exp}_g(\lambda;x_0)(x)=\operatorname{exp}_{g^C}(\lambda;x_0)(x)\operatorname{exp}_{g^B}(\lambda;x_0)(x).\qedhere\]  
\end{proof}
Note that, in fact, the formula~\eqref{formulaseries} was already suggesting equality~\eqref{32154000003265}. Since we have that $\operatorname{exp}_{g^C}(\lambda;x_0)(x)=e^{\lambda g_1(x)}$, the exponential series associated to $g$ converges if and only if the exponential series associated to $g^B$ converges. 

\begin{pro}
\label{expigmultpro}
Let $g:\mathbb{R}\to\mathbb{R}$ be a derivator such that $g^C=0$. Fix some $x_0\in \mathbb R$. Then
\begin{equation}\label{expigmult}
\operatorname{exp}_g(\lambda;x_0)(x)=\prod_{y\in [x_0,x)\cap D_g} (1+\lambda \Delta g(y))
\end{equation}
for all $x\geq x_0$.
\end{pro}
\begin{proof}
Fix any $x\geq x_0$. Let us first show the reasoning behind equality~\eqref{expigmult}. As a formal calculation, if we expand the product as if it was a polynomial on $\lambda$, we have
\[  1+\lambda\sum_{y\in [x_0,x)}\Delta g(y)+\lambda^2\sum_{\substack{y,s\in [x_0,x)\\ s<y}}\Delta g(y)\Delta g(s)+\lambda^3\sum_{\substack{t,y,s\in [x_0,x)\\ t<s<y}}\Delta g(y)\Delta g(s)\Delta g(t)+\cdots\]  
Recall Proposition~\ref{gmondiscontinua}. We observe that the coefficient associated to the $n^{\text{th}}$-power of $\lambda$ equals $\frac{g_n(x)}{n!}$. Hence, the above polynomial formally equals
\[  1+\lambda g_1(x)+\lambda^2\frac{g_2(x)}{2!}+\lambda^3\frac{g_3(x)}{3!}+\cdots\]  
which precisely is the $g$-exponential series. This proves equality~\eqref{expigmult} when $[x_0,x)\cap D_g$ is a finite set. Suppose then $[x_0,x)\cap D_g$ is infinite. We need to prove that the product~\eqref{expigmult} converges unconditionally, so let $\{t_n\}_{n\in\mathbb N}$ be any possible rearrangement of the elements of $[x_0,x)\cap D_g$. We show 
\[  \prod_{n=1}^k (1+\lambda \Delta g(t_n))\to\operatorname{exp}_g(\lambda;x_0)(x)\quad\text{as } k \text{ tends to infinity}.\]  
Fix any $\varepsilon>0$, there exists a natural $m$ such that
\[  \sum_{n=m}^{\infty}\left\lvert \lambda\right\rvert^n\frac{g_n(x)}{n!}<\frac{\varepsilon}{2}.\]  
Thanks to Proposition~\ref{gmondiscontinua}, there exists a natural $p$ such that for any $k\geq p$,
\[  
\left\lvert \sum_{\substack{s_1,\dots,s_n\in \{1,\dots,k\}\\ s_1<\dots<s_n}} \Delta g(t_{s_1})\cdots \Delta g(t_{s_n})-\frac{g_n(x)}{n!}\right\rvert<\frac{\varepsilon}{2(m-1)\left\lvert \lambda\right\rvert^n}
\]  
for all $n=1,\dots,m-1$. Hence, for all $k\geq p$,
\begin{equation*}
\begin{aligned}
&\left\lvert \prod_{n=1}^k (1+\lambda \Delta g(t_n))-\operatorname{exp}_g(\lambda;x_0)(x)\right\rvert\leq \sum_{n=1}^{\infty} \left\lvert \lambda\right\rvert^n\left\lvert \sum_{\substack{s_1,\dots,s_n\in \{1,\dots,k\}\\ s_1<\dots<s_n}} \Delta g(t_{s_1})\cdots \Delta g(t_{s_n})-\frac{g_n(x)}{n!}\right\rvert\\
\leq &\sum_{n=1}^{m-1} \left\lvert \lambda\right\rvert^n\left\lvert \sum_{\substack{s_1,\dots,s_n\in \{1,\dots,k\}\\ s_1<\dots<s_n}} \Delta g(t_{s_1})\cdots \Delta g(t_{s_n})-\frac{g_n(x)}{n!}\right\rvert+\sum_{n=m}^{\infty} \left\lvert \lambda\right\rvert^n\left\lvert \frac{g_n(x)}{n!}\right\rvert<\frac{\varepsilon}{2}+\frac{\varepsilon}{2}<\varepsilon.
\end{aligned}
\end{equation*}
and we obtain the result.
\end{proof}
\begin{rem}
\label{remexpigmult}
 Combining Proposition~\ref{expprodgcgb} and the series~\eqref{expigmultpro}, we get a formula of the exponential series for $x\geq x_0$. For any derivator $g$ and $x\geq x_0$,
\[  \operatorname{exp}_g(\lambda;x_0)(x)=e^{\lambda g_1^C(x)}\prod_{y\in [x_0,x)\cap D_g} (1+\lambda \Delta g(y)).\]  
In particular, $\operatorname{exp}_g(\lambda;x_0)(x)=0$ if and only if there exists some $y\in[x_0,x)\cap D_g$ such that $1+\lambda \Delta g(y)=0$. \par Take any $x<x_0$ such that $x\in\Omega_{x_0}$, from Corollary~\ref{omegas} and Proposition~\ref{expcambiopunto}, we can take $x_1\in\Omega_{x_0}$ such that $x_1<x<x_0$ and $1+\lambda\Delta g(s)\neq0$ for $s\in[x_1,x_0)$. Hence $\operatorname{exp}_g(\lambda;x_1)(x_0)\neq0$ and
\[  \operatorname{exp}_g(\lambda;x_0)(x)=\frac{\operatorname{exp}_g(\lambda;x_1)(x)}{\operatorname{exp}_g(\lambda;x_1)(x_0)},\]  
which means
\[  \operatorname{exp}_g(\lambda;x_0)(x)=e^{\lambda g_1^C(x)}\left(\prod_{y\in [x,x_0)\cap D_g} (1+\lambda \Delta g(y))\right)^{-1}.\]  

\end{rem}
\begin{thm}
	\label{0003431549687}
	If $1+\lambda\Delta g(x)\neq0$ for all $x<x_0$, there exists a Stieltjes-{analytic} extension of $\operatorname{exp}_g(\lambda;x_0)$ to the whole real line.
\end{thm}
\begin{proof}
	Assume without loss of generality that $x_0=0$. Let $x\in\mathbb{R}$, take $t<\min\{x,0\}$, define
	\begin{equation}
		\label{3215421000654168461}
		\operatorname{Exp}_g(\lambda;0)(x):=\frac{\operatorname{exp}_g(\lambda;t)(x)}{\operatorname{exp}_g(\lambda;t)(0)}. 
	\end{equation}
	Since $t<\min\{x,0\}$ and $1+\lambda\Delta g(s)\neq0$ for all $s<0$, following Remark~\ref{remexpigmult}, the exponential series centered at $t$ takes a nonzero value at $0$. Besides, for all $y\in\Omega_0$,
	\[  
	\operatorname{exp}_g(\lambda;t)(0)\operatorname{exp}_g(\lambda;0)(y)=\operatorname{exp}_g(\lambda;t)(y),
	\]  
	and, hence, expression~\eqref{3215421000654168461} functions as an extension of $\operatorname{exp}_g(\lambda;0)$. Let us see that $\operatorname{Exp}_g(\lambda;0)(x)$ is well defined. Choose $t'\in\mathbb{R}$ such that $t'<\min\{x,0\}$. Assume without loss of generality that $t<t'$, then
	\[  
	\operatorname{exp}_g(\lambda;t)(t')\operatorname{exp}_g(\lambda;t')(y)=\operatorname{exp}_g(\lambda;t)(y)
	\]  
	for any $y\geq t'$. Hence
	\[  
	\frac{\operatorname{exp}_g(\lambda;t')(x)}{\operatorname{exp}_g(\lambda;t')(0)}=\frac{\operatorname{exp}_g(\lambda;t)(t')\operatorname{exp}_g(\lambda;t')(x)}{\operatorname{exp}_g(\lambda;t)(t')\operatorname{exp}_g(\lambda;t')(0)}=\frac{\operatorname{exp}_g(\lambda;t)(x)}{\operatorname{exp}_g(\lambda;t)(0)},
	\]  
	since $\operatorname{exp}_g(\lambda;t)(t')\neq0$. 
\end{proof}
\begin{cor}
Following Remark~\ref{remexpigmult} and expression~\eqref{3215421000654168461}, we have that,
\[  \operatorname{Exp}_g(\lambda;x_0)(x)=\begin{dcases}e^{\lambda g_1^C(x)}\prod_{y\in [x_0,x)\cap D_g} (1+\lambda \Delta g(y)), & x\geq x_0,\\ e^{\lambda g_1^C(x)}\left(\prod_{y\in [x,x_0)\cap D_g} (1+\lambda \Delta g(y))\right)^{-1}& x<x_0. \end{dcases}\]  
\end{cor}
\begin{thm}
	\label{solecder}
	For any $s\in\mathbb{R}$, the function $\operatorname{exp}_g(\lambda;s)$ solves the differential equation 
	\begin{equation}
		\label{000000000313132223544}
		\begin{dcases}
			v'_g(x)=\lambda v(x), & \forall x\in (s-\delta,\infty),\\
			v(s)=1,
		\end{dcases}
	\end{equation}
	for all $\delta>0$ such that $(s-\delta,\infty)\subset \Omega_s$, {where $\Omega_{s}$ is the maximal interval where 
	 $ 
		\sum_{n=0}^\infty \lambda^n \frac{g_{n}(x)}{n!}
		$ 
		converges absolutely.}
\end{thm}
\begin{proof}
	We know a $\delta>0$ such that $(s-\delta,\infty)\subset \Omega_s$ exists. Given $a,b \in (s-\delta,\infty)$ such that $a<b$, from Proposition~\ref{seriesint},
	\begin{align*}
		\int_{[a,b)}\operatorname{exp}_g(\lambda;s)\operatorname{d}\mu_g &= \sum_{n=0}^\infty \lambda^{n}\frac{g_{s,n+1}(b)}{(n+1)!}- \sum_{n=0}^\infty \lambda^{n}\frac{g_{s,n+1}(a)}{(n+1)!}\\
		&=\frac{1}{\lambda}\left(\operatorname{exp}_g(\lambda;s)(b)-\operatorname{exp}_g(\lambda;s)(a)\right).
	\end{align*}
	Then
	\[  \operatorname{exp}_g(\lambda;s)(b)-\operatorname{exp}_g(\lambda;s)(a)= \int_{[a,b)}\lambda\operatorname{exp}_g(\lambda;s)\operatorname{d}\mu_g. \]  
	From Proposition~\ref{primder}, $\operatorname{exp}_g(\lambda;s)'_g(x)=\lambda \operatorname{exp}_g(\lambda;s)(x)$ for all $x\in(a,b)-C_g$. Repeating the argument with $a$ and $b$ approaching the extremes of the interval,
	\[  \operatorname{exp}_g(\lambda;s)'_g(x)=\lambda \operatorname{exp}_g(\lambda;s)(x)\quad\forall x\in(s-\delta,\infty).\qedhere\]  
\end{proof}

The Stieltjes-analytic extension of the exponential series given in Theorem~\ref{0003431549687} is a solution of the differential equation~\eqref{000000000313132223544} as well. We have then the following corollaries.

\begin{cor}
\label{expcinf}
If $1+\lambda\Delta g(x)\neq0$ for all $x<x_0$, $\operatorname{Exp}_g(\lambda;x_0)$ is a Stieltjes-analytic solution defined on the whole real line of the problem
	\begin{equation*}
		\begin{dcases}
			v'_g(x)=\lambda v(x), & x\in\mathbb R\\
			v(x_0)=1.
		\end{dcases}
	\end{equation*}
\end{cor}

\begin{cor}[Exponential series]
$\operatorname{Exp}_g(1;0)$ is a Stieltjes-analytic solution defined on the whole real line of the problem
	\begin{equation*}
		\begin{dcases}
			v'_g(x)=v(x), & x\in\mathbb R\\
			v(0)=1.
		\end{dcases}
	\end{equation*}
\end{cor}
\begin{dfn}
Let $g:\mathbb{R}\to\mathbb{R}$ be a derivator. Fix any $\lambda\in\mathbb R$ such that $1+\lambda \Delta g(x)\neq0$ for all $x\in\mathbb R$. We define $\operatorname{Exp}_g^\lambda\equiv\operatorname{Exp}_g(\lambda;0)$ as the exponential function associated with $\lambda$ and $g$. We call $\operatorname{Exp}_g\equiv\operatorname{Exp}_g^1$ the $g$-exponential function.
\end{dfn}
Since $\operatorname{Exp}_g^\lambda$ does not vanish we have
\[  \frac{\operatorname{Exp}_g^\lambda}{\operatorname{Exp}_g^\lambda(x_0)}=\operatorname{Exp}_g(\lambda;x_0)\]  
for all $x_0\in\mathbb R$. From Corollary~\ref{expcinf},
\[  \operatorname{Exp}_g^\lambda\in\mathcal{C}^\infty_g(\mathbb R).\]  

\begin{rem}
	Let us now look at the relationship between the exponential series and the exponential function. Both \cite{RodrigoTheory} and \cite{Onfirst} work with the differential equation
	\begin{equation}
		\label{eqdeflit}
		\begin{dcases}
			v'_g(x)=\beta (x) v(x), & \forall x\in [0,T)-C_g,\\
			v(0)=1,
		\end{dcases}
	\end{equation}
	where $T>0$, and $\beta\in\mathcal{L}_g^1([0,t),\mathbb{F})$ such that $1+\beta (x)\Delta g(x)\neq0$ for all $x\in[0,T)$. In \cite[Lemma~6.3]{RodrigoTheory} and \cite[Theorem~4.2]{Onfirst} an explicit solution is computed and called \emph{exponential function}. In fact, thanks to \cite[Theorem 7.3]{RodrigoTheory} the uniqueness and existence of the solution is guaranteed.

	Note how the hypothesis $1+\beta\Delta g\neq0$ appears. It essentially guarantees that the solution of~\eqref{eqdeflit} does not vanish at a certain discontinuity point of $g$, but it is not needed to compute the solution or to guarantee existence and uniqueness. See how that relates to Remark~\ref{remexpigmult}, as it is precisely stating the same. From Theorem~\ref{solecder}, we know that $\operatorname{exp}(\lambda;0)\in\mathcal{AC}_g([0,T],\mathbb{F})$ solves equation~\eqref{eqdeflit} for $\beta=\lambda$, so the exponential series and the exponential function defined on \cite{RodrigoTheory} and \cite{Onfirst} match for $x\geq 0$. 

	The computations made on Proposition~\ref{expigmultpro} were already shown in \cite[Theorem 4.2]{Onfirst}, as they calculated a explicit solution of~\eqref{eqdeflit}. Notice that formula~\eqref{32154000003265} is proven in \cite[Theorem 4.2]{Onfirst} as well. They say the solution of~\eqref{eqdeflit} is the solution of the same differential equation associated with $g^C$ multiplied by the solution associated with $g^B$. We managed to reconstruct all of this results independently, basing our proofs on properties of Stieltjes-analytic functions and $g$-monomial series.
\end{rem}

\subsection{Higher order linear Stieltjes differential equations with constant coefficients}
\label{higherorder}

Let us go back to the linear differential problem~\eqref{01111} and its associated difference equation~\eqref{011111}. Clearly, there is a bijection between the solutions of 
\begin{equation}
	\label{0111111}
	\begin{dcases}
		b_{n+m}=\sum_{k=0}^{m-1}\lambda_k b_{n+k}, & \lambda_k\in\mathbb{F}, \ n\geq0,\\
		b_k=c_k,& c_k\in\mathbb{F}, \ k\in\{0,\dots,m-1\}
	\end{dcases}
\end{equation}
and solutions of problem~\eqref{011111}, just by taking $b_n=a_nn!$. In fact, this bijection is a linear transformation, both the spaces of solutions of problems~\eqref{011111} and~\eqref{0111111} are vector spaces, for more details see \cite[Chapter~3]{DE}. Let us bound these sequences.
\begin{lem}
\label{lemasequences}
If $\{b_n\}_{n\in\mathbb N}$ solves problem~\eqref{0111111}, then there exists a constant $M>0$ such that
$\left\lvert b_n\right\rvert\leq M^{n+1}$ 
for all $n\geq0$.
\end{lem}
\begin{proof}
Note first that $m\geq1$. Let $C=\max\{1,\left\lvert c_0\right\rvert,\dots,\left\lvert c_{m-1}\right\rvert,\left\lvert \lambda_0\right\rvert,\dots,\left\lvert \lambda_{m-1}\right\rvert\}$ and take $M=mC$. Since $M\geq1$, $\left\lvert b_k\right\rvert\leq M^{k+1}$ for $k\in\{0,\dots,m-1\}$. Let us show that if $n\in\mathbb N$ is such that $n\geq m-1$ and ${\left\lvert b_k\right\rvert\leq M^{k+1}}$ for $k\leq n$, then ${\left\lvert b_{n+1}\right\rvert\leq M^{n+2}}$. Since $n\geq m-1$, we have that
\[  b_{n+1}=\lambda_{m-1}b_{n}+\cdots+\lambda_{0}b_{n-(m-1)},\]  
so
\[  \left\lvert b_{n+1}\right\rvert\leq\left\lvert \lambda_{m-1}\right\rvert\left\lvert b_{n}\right\rvert+\cdots+\left\lvert \lambda_{0}\right\rvert\left\lvert b_{n-(m-1)}\right\rvert\leq mC M^{n+1}=M^{n+2}.\]  
Applying induction we have the result.\qedhere
\end{proof}
\begin{cor} \label{boundseq}
Let $\{a_n\}_{n\in\mathbb N}$ be the solution of problem~\eqref{01111}. Since $\{a_nn!\}_{n\in{\mathbb N}}$ solves problem~\eqref{0111111}, applying Lemma~\ref{lemasequences}, there exists some constant $M>0$ such that
\[  \left\lvert a_n\right\rvert\leq \frac{M^{n+1}}{n!},\]  
for all $n\geq0$.
\end{cor}

This result is key to prove convergence of the $g$-monomial series. It is basically telling us that the series converges at least on the same interval of some exponential, see Corollary~\ref{omegas}. We can now prove that the Stieltjes-analytic function defined by the solution of problem~\eqref{01111} solves the original problem~\eqref{01111}.

\begin{thm}
Let $\{a_n\}_{n\in\mathbb N}$ be the solution of problem~{\eqref{011111}}. Then
\[  v(x)=\sum_{n=0}^{\infty} a_n g_n(x)\]  
converges absolutely on $(t,+\infty)$ for some $t<x_0$, is such that $v\in \mathcal{C}^{\infty}_g((t,+\infty),{\mathbb F})$ and solves problem~\eqref{01111}.
\end{thm}
\begin{proof}
Thanks to Corollary~\ref{boundseq}, there exists $M>0$ such that $\left\lvert a_n\right\rvert\leq \frac{M^{n+1}}{n!}$, for all $n\geq0$. For all $k\in\mathbb N$, the series
\[  \sum_{n=k}^{\infty}\left\lvert a_n\right\rvert\frac{n!}{(n-k)!}\left\lvert g_{n-k}(x)\right\rvert\leq\sum_{n=k}^{\infty}\frac{M^{n+1}}{(n-k)!}\left\lvert g_{n-k}(x)\right\rvert\leq M^{k+1} \sum_{n=0}^{\infty}\frac{M^n}{n!}\left\lvert g_{n}(x)\right\rvert\]  
converges on $(t,+\infty)$ for some $t<x_0$. That means if we apply Propositions~\ref{seriesint} and~\ref{primder} recursively we have that
\begin{equation}\label{fooorrlrl}
v^{(k)}_g(x)=\sum_{n=k}^{\infty}a_n\frac{n!}{(n-k)!}g_{n-k}(x)\end{equation}
for all $x\in(t,+\infty)$ and all $k\geq0$. Note we are assuming $\infty\notin N^+_g$ so we can $g$-differentiate $v^{(k)}_g$ at all points of $(t,+\infty)$. This means $v\in \mathcal{C}^{\infty}_g((t,+\infty),{\mathbb F})$. Now, from equation~\eqref{fooorrlrl}, for all $k\geq0$,
\[  v^{(k)}_g(x_0)=a_k k!.\]  
Therefore, $v^{(k)}_g(x_0)=c_k$ for all $k\in\{0,\dots,m-1\}$ since $\{a_n\}_{n\in\mathbb N}$ solves problem~\eqref{01111}. Besides, for all $x\in(t,+\infty)$,
\[  \begin{aligned}
\sum_{k=0}^{m-1}\lambda_k v^{(k)}_g(x)&=\sum_{k=0}^{m-1}\lambda_k \sum_{n=0}^{\infty}a_{n+k}\frac{(n+k)!}{n!}g_{n}(x)=\sum_{n=0}^{\infty}\left(\sum_{k=0}^{m-1}\lambda_k a_{n+k}\frac{(n+k)!}{n!}\right)g_{n}(x)\\&=\sum_{n=0}^{\infty}a_{n+m}\frac{(n+m)!}{n!}g_{n}(x)= v^{(m)}_g(x).
\end{aligned}
\]  
Hence, $v$ solves problem~\eqref{01111}.
\end{proof}

Take a look now at the nonhomogeneous case. Fix some $x_0\in\mathbb R$. Consider the following problem:
\begin{equation}
	\label{nonh}
	\begin{dcases}
		v^{(m)}_g(x)=\sum_{k=0}^{m-1}\lambda_k v^{(k)}_g(x)+f(x), & \lambda_k\in\mathbb{F},\\
		v^{(k)}_g(x_0)=c_k,& c_k\in\mathbb{F}, \ k\in\{0,\dots,m-1\}.
	\end{dcases}
\end{equation}
Suppose we can write $f$ as a $g$-monomial series centered at $x_0$. Then,
\[  f(x)=\sum_{n=0}^\infty r_n g_n(x),\]  
for some $\{r_n\}_{n\in\mathbb N}\subset \mathbb R$. We obtain the following difference equation matching coefficients:
\begin{equation}
	\label{dfeqnonh}
	\begin{dcases}
		a_{n+m}(n+m)!=\sum_{k=0}^{m-1}\lambda_k a_{n+k}(n+k)!+r_nn!, & \lambda_k\in\mathbb{F}, \ n\geq0,\\
		a_kk!=c_k,& c_k\in\mathbb{F}, \ k\in\{0,\dots,m-1\}.
	\end{dcases}
\end{equation}
Analogously, there is a bijection between the solutions of
\begin{equation}
	\label{dfeqnonhalt}
	\begin{dcases}
		b_{n+m}=\sum_{k=0}^{m-1}\lambda_k b_{n+k}+r_nn!, & \lambda_k\in\mathbb{F}, \ n\geq0,\\
		b_k=c_k,& c_k\in\mathbb{F}, \ k\in\{0,\dots,m-1\},
	\end{dcases}
\end{equation}
and solutions of problem~\eqref{dfeqnonh}. As we did in Lemma~\ref{lemasequences}, we will bound again these sequences, as it is needed for the convergence of the $g$-monomial series.

\begin{lem}
\label{lmaolol}
If $\{b_n\}_{n\in\mathbb N}$ solves problem~\eqref{dfeqnonhalt}, then there exists a constant $M>0$ such that
\[  \left\lvert b_{n+m}\right\rvert\leq \sum_{k=0}^{n}M^{k}R_{n-k}\]  
for all $n\geq0$, where $R_n=\left\lvert r_n\right\rvert n!$ for all $n\geq1$ and $R_0>0$ is some positive number.
\end{lem}
\begin{proof}
Let $\{b_n\}_{n\in\mathbb N}$ be the solution to problem~\eqref{dfeqnonhalt}. Take $C=\max\{1,\sum_{k=0}^{m-1}\left\lvert \lambda_k\right\rvert,\sum_{k=0}^{m-1}\left\lvert b_k\right\rvert\}$ and define $M=C^2\geq C$. Since $M\geq1$, the powers of $M$ form a nondecreasing sequence. We have that,
\[  \left\lvert b_m\right\rvert\leq\sum_{k=0}^{m-1} \left\lvert \lambda_k\right\rvert \left\lvert b_k\right\rvert+\left\lvert r_0\right\rvert\leq C^2+\left\lvert r_0\right\rvert=M+\left\lvert r_0\right\rvert.\]  
Define $R_0=M+\left\lvert r_0\right\rvert$. Let us bound $b_{m+1}$ to show how the process continues. Note that $\left\lvert b_k\right\rvert\leq M+\left\lvert r_0\right\rvert=R_0$ for $k=1,\dots,m$. So
\[  \left\lvert b_{m+1}\right\rvert\leq\sum_{k=0}^{m-1} \left\lvert \lambda_k\right\rvert \left\lvert b_{k+1}\right\rvert+\left\lvert r_1\right\rvert\leq MR_0+\left\lvert r_1\right\rvert= MR_0+R_1.\]  
Now, we have that $\left\lvert b_k\right\rvert\leq MR_0+R_1$ for $k=2,\dots,m+1$, hence, we can bound the following coefficient and so on. Let us apply induction, let $n\in\mathbb N$ be such that
\[  \left\lvert b_{i+m}\right\rvert\leq \sum_{k=0}^{i}M^{k}R_{i-k}\]  
for all $i\leq n$. By definition of $M$ and $R_k$, $\left\lvert b_i\right\rvert\leq\sum_{k=0}^{n}M^{k}R_{n-k}$ for all $i\leq n$. Then
\[  \left\lvert b_{n+1+m}\right\rvert\leq \sum_{k=0}^{m-1} \left\lvert \lambda_k\right\rvert \left\lvert b_{k+1}\right\rvert+\left\lvert r_{n+1}\right\rvert(n+1)!\leq M\left(\sum_{k=0}^{n}M^{k}R_{n-k}\right)+R_{n+1}=\sum_{k=0}^{n+1}M^{k}R_{n+1-k}.\qedhere\]  
\end{proof}

\begin{thm}
\label{eqhomth}
Let $\{a_n\}_{n\in\mathbb N}$ be the solution of problem~\eqref{dfeqnonh}. {Assume}
\begin{equation}\label{serie f}f(x)=\sum_{n=0}^{\infty} r_n g_n(x)\end{equation}
converges absolutely {on $[x_0,c]$ for some $c>x_0$. Suppose $x_0\notin N^{-}_g$ and $c\notin D_g\cup N^{+}_g\cup C_g$, so the $g$-derivative is well defined at all points of $[x_0,c]$}, then
\begin{equation}\label{serie v}v(x)=\sum_{n=0}^{\infty} a_n g_n(x)\end{equation}
converges absolutely on $[x_0,c]$, is such that $v\in\mathcal{C}^{m}_g([x_0,c],{\mathbb F})$ and solves problem~\eqref{nonh}.
\end{thm}
\begin{proof}
Let $\{a_n\}_{n\in\mathbb N}$ be the solution of problem~\eqref{dfeqnonh}. Thanks to Lemma~\ref{lmaolol}, there exists $M>0$ such that \[  \left\lvert a_{n+m}\right\rvert\leq \frac{1}{(n+m)!}\sum_{k=0}^{n}M^{k}R_{n-k},\]  for all $n\geq0$, where $R_n=\left\lvert r_n\right\rvert n!$ for all $n\geq1$ and $R_0>0$ is some positive number. Denote $\check r_n=\frac{R_n}{n!}$ for all $n\geq0$. We are going to prove that the $g$-monomial series the sequence $\{a_n\}_{n\in\mathbb N}$ defines converges absolutely on $[x_0,c)$. For any $x\in[x_0,c]$, from the product formula for absolutely convergent series, the following series converges:
\begin{equation}\label{pepe}
\left(\sum_{n=0}^\infty M^n \frac{\left\lvert g_{n}(x)\right\rvert}{n!}\right)\left(\sum_{n=0}^\infty \check r_n\left\lvert g_n(x)\right\rvert\right)=\sum_{n=0}^\infty \sum_{k=0}^n \frac{M^{k}}{k!}\check r_{n-k}\left\lvert g_{k}(s)\right\rvert\left\lvert g_{n-k}(x)\right\rvert<\infty,\end{equation}
note that $\check r_n$ and $\left\lvert r_n\right\rvert$ define $g$-monomial series that converge at the same points. Then, for any $x\in[x_0,c]$, we have that
\begin{equation}\label{pepe2}\begin{aligned}
\sum_{n=0}^{\infty}\left\lvert a_{n+m}\right\rvert\frac{(n+m)!}{n!}\left\lvert g_{n}(x)\right\rvert&\leq\sum_{n=0}^{\infty}\sum_{k=0}^{n}\frac{M^{k}R_{n-k}}{n!}\left\lvert g_{n}(x)\right\rvert=\sum_{n=0}^{\infty}\sum_{k=0}^{n}\frac{M^{k}\check r_{n-k}(n-k)!}{n!}\left\lvert g_{n}(x)\right\rvert\\&=\sum_{n=0}^{\infty}\sum_{k=0}^{n}\frac{1}{{n\choose k}}\frac{M^{k}}{k!}\check r_{n-k}\left\lvert g_{n}(x)\right\rvert\leq\sum_{n=0}^{\infty}\sum_{k=0}^{n}\frac{M^{k}}{k!}\check r_{n-k}\left\lvert g_{n}(x)\right\rvert<\infty,
\end{aligned}
\end{equation}
applying Proposition~\ref{cotasupder} and equation~\eqref{pepe}. Applying now Propositions~\ref{seriesint} and~\ref{primder} $m$ consecutive times, we have that 
\[  v^{k}_g(x)=\sum_{n=0}^{\infty}a_{n+k}\frac{(n+k)!}{n!}g_{n}(x)\]  
for all $x\in[x_0,c]$ and $k\in\{0,\dots,m\}$, with those series converging absolutely on $[x_0,c]$. Hence $v^{(k)}_g(x_0)=a_k k!=c_k$ for all $k\in\{0,\dots,m-1\}$. Besides, for all $x\in[x_0,c]$,
\[  \begin{aligned}
\sum_{k=0}^{m-1}\lambda_k v^{(k)}_g(x)+f(x)&=\sum_{k=0}^{m-1}\lambda_k \sum_{n=0}^{\infty}a_{n+k}\frac{(n+k)!}{n!}g_{n}(x)+\sum_{n=0}^{\infty}r_{n}g_n(x)\\&=\sum_{n=0}^{\infty}\left(\sum_{k=0}^{m-1}\lambda_k a_{n+k}\frac{(n+k)!}{n!}+r_n\right)g_{n}(x)=\sum_{n=0}^{\infty}a_{n+m}\frac{(n+m)!}{n!}g_{n}(x)= v^{(m)}_g(x).
\end{aligned}
\]  
Thus, $v$ solves problem~\eqref{nonh}. \qedhere
\end{proof}
\begin{rem}
We can ensure convergence of the $g$-monomial series~\eqref{serie v} at points to the left of $x_0$ if the series~\eqref{serie f} converges at points to the left of $x_0$ as well. Suppose the $g$-monomial series~\eqref{serie f} converges absolutely on $[c_1,c_2]$ for some $c_1\in(-\infty,x_0)$ and $c_2\in(x_0,+\infty)$. From Theorem~\ref{eqhomth}, the series~\eqref{serie v} converges absolutely on $[x_0,c_2]$. We have two cases:
\begin{enumerate}[noitemsep, itemsep=.1cm]
		\item[~1.] $\Delta g(x)=0$ for all $x\in[c_1,x_0)$. Hence $g_n(x)=g_1(x)^n$ and $\left\lvert g_{n-k}(x)\right\rvert\left\lvert g_k(x)\right\rvert=\left\lvert g_n(x)\right\rvert$ for all $x\in[c_2,x_0]$. Thus, equation~\eqref{pepe} implies the absolute convergence of the $g$-monomial series~\eqref{pepe2} for all $x\in[c_2,c_1]$. Therefore, the series~\eqref{serie v} converges absolutely and solves problem~\eqref{nonh} on $[c_2,c_1]$.
		\item[~2.] There exists some $x_1\in[c_1,x_0)$ such that $\Delta g(x_1)\neq0$. That means, thanks to Proposition~\ref{sucesionizq}, that there exists some constant $M>0$ such that $\left\lvert r_n\right\rvert n!\leq M^n$ for all $n\geq0$. Applying Lemma~\ref{lmaolol} we have that there exists another constant $C>0$ such that $\left\lvert a_n\right\rvert n!\leq C^n$ for all $n\geq0$. Hence, there exists some $t<x_0$ such that both $g$-monomial series~\eqref{serie v} and~\eqref{serie f} converge absolutely on $(t,+\infty)$. Thus, $v$ is well defined and solves problem~\eqref{nonh} on $(t,+\infty)$.
\end{enumerate}
\end{rem}

\section*{Acknowledgments}
The authors were partially supported by Xunta de Galicia, project ED431C 2019/02, and by the Agencia Estatal de Investigaci\'on (AEI) of Spain under grant MTM2016-75140-P, co-financed by the European Community fund FEDER.

\section*{Competing interests}
The authors declare no competing interests.

\bibliographystyle{spmpsciper}
\bibliography{SA}

 \end{document}